\documentclass[12pt,reqno]{amsart}
\usepackage{mypub}

\begin{document}

\title[Local existence of solutions to the Euler--Poisson system]
{Local existence of solutions to the Euler--Poisson system,
  including densities without compact support}


\author[U. Brauer]{Uwe Brauer}

\address{%
Uwe Brauer 
  Departamento de  Matemática Aplicada\\ Universidad Complutense Madrid
28040 Madrid, Spain}
\email{oub@mat.ucm.es}

\thanks{U.~B. gratefully acknowledges support from  grants MTM2012-31298
and MTM2016-75465 from MINECO, Spain}
   
\author[L. Karp]{Lavi Karp}

\thanks{Accepted for publication in \textbf{Journal of Differential Equations}}	 

\address{%
Lavi Karp
Department of Mathematics\\ ORT Braude College\\
P.O. Box 78, 21982 Karmiel\\ Israel}

\email{karp@braude.ac.il}



\begin{abstract}
  Local existence and well posedness for a class of solutions for the
  Euler Poisson system is shown. 
  These solutions have a density $\rho$ which either falls off at
  infinity or has compact support. 
  The solutions have finite mass, finite energy functional and include
  the static spherical solutions for $\gamma=\frac{6}{5}$. 
  The result is achieved by using weighted Sobolev spaces of
  fractional order and a new non linear estimate which allows to
  estimate the physical density by the regularised non linear matter
  variable.
  Gamblin also has studied this setting but using very different
  functional spaces. 
  However we believe that the functional setting we use is more
  appropriate to describe a physical isolated body and more suitable
  to study the Newtonian limit.

\end{abstract}

\maketitle{}

\section{Introduction}
\label{sec:iteration-scheme}

We consider the Euler--Poisson system
\begin{align}
  \label{eq:iteration-scheme:1}
  \partial_t\rho + v^a\partial_a\rho + \rho \partial_a v^a
  &= 0 \\
  \label{eq:iteration-scheme:2}
  \rho \left(\partial_t v^a + v^b\partial_b v^a \right) + \partial^a
  p &=  - \rho\partial^a\phi  \\
  \label{eq:iteration-scheme:3}
  \Delta\phi &= 4\pi G \rho 
\end{align}
where $G$ denotes the gravitational constant. 
Using suitable physical units we can set $G=1$. 
Here we have used the summation convention, for example,
$v^k\partial_k:=\sum_{k=1}^3v^k\partial_k$, a convention we will use in the
rest of the paper wherever it seems appropriate to us.
Moreover $\partial^a\phi:=\delta^{ab}\partial_b\phi$, and we will wherever it is
convenient to denote $ \partial^a\phi$ by $\nabla\phi$.
In this paper we consider the barotropic equation of state
\begin{equation}
  \label{eq:state}
  p=K\rho^\gamma\quad 1<\gamma,\  0<K,
\end{equation}
and we  study this system with initial data for the density which either
has compact support or falls off at infinity in an appropriate way.
It is well known that the usual symmetrization of the Euler equations
is badly behaved in such cases. 
The coefficients of the system degenerate or become unbounded when
$\rho$ approaches zero. 
It was observed by Makino
\cite{makino86:_local_exist_theor_evolut_equat_gaseous_stars} that 
this difficulty can be to some extend circumvented by using a new
matter variable $w$ in place of the density. 
For this reason, we introduce the quantity
\begin{equation}
  \label{eq:proto-euler-poisson-banach:1}
  w=\frac{2\sqrt{K \gamma}}{\gamma-1}\rho^{\frac{\gamma-1}{2}},
\end{equation}
which allows to treat the situation where $\rho=0$. 
Replacing the density $\rho$ by the Makino variable $w$, the system
(\ref{eq:iteration-scheme:1})--(\ref{eq:iteration-scheme:3}) coupled with the 
equation of state (\ref{eq:state}) takes the
following form:
\begin{align}
  \label{eq:iteration-scheme:1b}
  \partial_tw + v^a\partial_aw + \frac{\gamma-1}{2}w \partial_a v^a
  &= 0 \\
  \label{eq:iteration-scheme:2b}
  \partial_t v^a + v^b\partial_b v^a  + \frac{\gamma-1}{2}w\partial^a
  w &=  -          \partial^a\phi  \\
  \label{eq:iteration-scheme:3b}
  \Delta\phi &= 4\pi\rho 
\end{align}
which we will sometimes denote as the  Euler--Poisson--Makino system.
The Euler--Poisson system consists of a hyperbolic system of evolution
equations and the elliptic Poisson equation.

Traditionally symmetric hyperbolic systems have been solved in
Sobolev spaces $H^s$ because their norm allows in a convenient
way to obtain energy estimates.
But there are situations in which these spaces are too restrictive 
\cite{KATO,
majda84:_compr_fluid_flow_system_conser
}. One of them is the Euler-Poisson equations when the density has no compact
support. 
We therefore treat the Euler--Poisson system in a new functional
setting which involves weighted Sobolev spaces of fractional order.
These spaces have been introduced by
Triebel~\cite{triebel76:_spaces_kudrj2} and have been used successfully by the
authors~\cite{BK7,
  BK8
} to prove a similar result for the Einstein--Euler system.
In this setting, we prove local existence, uniqueness and well
posedness of classical solutions.
We mean by well posedness  in this paper the conservation of the 
initial regularity, but not the continuity of the flow map.

The benefit of these spaces is that they enable us to consider a wide
range of $\gamma$ for the equation of state $p=K\rho^{\gamma}$, and also to
construct solutions with a couple of interesting features, which we
will discuss briefly in the following.
The solutions we obtain include densities without compact support but
with finite mass and energy functional for $1<\gamma <\frac{5}{3}$. 
In particular, they include static spherical solutions with finite
mass but infinite radius if $\gamma=\frac{6}{5}$.

An essential ingredient of the proof is a new nonlinear estimate of a
power of functions in the weighted fractional Sobolev spaces,
Proposition \ref{prop:power}. 
This estimate enables us to obtain a solution of the Poisson equation
(\ref{eq:iteration-scheme:3b}) for densities without compact support
in terms of the Makino variable $w$.

The problem studied here has been treated already by Gamblin
\cite{gamblin93:_solut_euler
} and Bezard
\cite{Bezard_93
}, so we compare their results with the one obtained by us.
The main differences are the choice of the functional spaces that are
used to prove their results, and the properties of the corresponding
solutions.
Bezard uses the ordinary Sobolev spaces $H^s$, and therefore
his claim that his solutions include static spherical solutions if
$\gamma=\frac{6}{5}$ is simply not correct, since the initial data of the
corresponding Makino variable do not belong to $H^s$. 
Gamblin uses the uniformly locally Sobolev spaces $H_{ul}^s$, which
have been introduced by Kato
\cite{KATO
}. 
This type of spaces includes bounded functions as $|x|\to \infty$, and hence
Gamblin's solutions contain spherical static solutions, however the
use of these spaces is problematic in several important aspects which
we list here shortly, for details we refer to Subsection
\ref{subsection:advantages}.
First it should be noted that there is no well posedness results in
the $H^{s}_{ul}$ spaces, as it was pointed out by Majda \cite[Thm 2.1
p50]{majda84:_compr_fluid_flow_system_conser
}. 
By this we mean that for given initial data $u_0\in H^s_{ul}$, the
corresponding solution belongs only to
$C \left([0,T];H_{loc}^s \right)\cap C^1 \left([0,T];H_{loc}^{s-1}
\right)$.

Thus Gamblin's solutions face these disadvantages. 
Moreover, in his setting the density $\rho$ belongs to the Sobolev space
$W^{1,p}$ ($1\leq p <3$), while the velocity $v^a\in H^s_{ul}$, and hence the
density falls off to zero at infinity but velocity does not. 
Such behavior of the solutions disagrees with the physical
interpretations of the model of isolated bodies. 
As we mentioned above, the Sobolev spaces $H^s$ are the most
convenient spaces for quasi linear first order symmetric hyperbolic
systems. 
But there are various circumstances in which either the initial data,
or the coefficients of the system do not belong to this class, for
example the asymptotically flat spacetime in general relativity
\cite{Christodoulou_1981
}, or if the density belongs to $L^\infty$
\cite{majda84:_compr_fluid_flow_system_conser
}, and of course the Euler--Poisson for densities without compact
support \cite{gamblin93:_solut_euler
} Therefore we suggest a different approach, namely, we establish
well--posedness of quasi linear symmetric hyperbolic systems in the
weighted fractional Sobolev spaces, see subsection
\ref{sec:symm-hyperb-syst}. 
This approach suits several situations where the initial data and the
coefficients do not belong to the $H^s$ spaces, and in particular it
can be applied to coupled hyperbolic--elliptic systems such as the
Euler-Poisson.

Finally, we want to emphasize that we only consider local in time
solutions. 
Due to the complicated nonlinear character of the Euler equations, it
is expected that classical solution will break down in a finite time
interval see for example
\cite{sideris85:_format_singul_three
}. 
Global in time existence results can only be expected if the tendency
to form singularities is somehow compensated. 
For example, Guo
\cite{guo98:_smoot_euler_poiss_r
} considered a simple two-fluid model to describe the dynamics of a
plasma is the Euler–Poisson system, where the compressible electron
fluid interacts with its own electric field against a constant charged
ion background. 
This feature results in a repulsive force and enables the proof of
global existence.
Another example is the Euler--Poisson system with a cosmological
constant
\cite{ICH6
}, where initial data are considered which are small and describe
deviations of a given exponentially expanding background solution,
and which lead to classical solutions that exist globally in time.

\subsection{Structure of the proof and organization of the paper}
\label{sec:struct-proof-organ}
The most obvious way to solve system
(\ref{eq:iteration-scheme:1b})--(\ref{eq:iteration-scheme:3b}) would be
to apply some sort of iteration procedure or a fixed--point argument
directly to that system. 
But since the system is coupled to an elliptic equation, it seemed
more convenient and transparent to split up the proof in several
parts. 
Firstly we prove local existence and well posedness for a general
symmetric hyperbolic system (with $A^0=\mathds{1}$) in the weighted
Sobolev spaces.

Since the density falls off but could become zero, we will need the
established tool of regularizing the system, by introducing a new
matter variable, the Makino variable
(\ref{eq:proto-euler-poisson-banach:1}).
In this setting the power $w^{\frac{2}{\gamma-1}}$ must be estimated in the
weighted fractional norm. 
The estimates of the power in the $H^s$ spaces under certain
restrictions on the power and $s$ are known (see e.g. 
\cite{runst96:_sobol_spaces_fract_order_nemyt}). 
An essential ingredient of our proof is a nonlinear power estimate in
the weighted fractional Sobolev spaces that preserves the regularity
and improves the fall off at infinity (Proposition \ref{prop:power}). 
It enables us to apply the known estimates for the Poisson equation
(\ref{eq:iteration-scheme:3b}) in these spaces. 
We then prove the existence of solutions to the Euler--Poisson--Makino
system by using a fixed--point argument.
In any case, either for the fixed--point or for the direct iteration we
are faced with the well known fact that we have to use a
\textit{higher} and a \textit{lower} norm.
We show boundness in the higher norm and contraction in the lower. 
Under this circumstances the existence of a fixed point in the higher
norm is well known.
However, we have not found such a modified fixed--point theorem in the
literature, and that is why we have added it together with its proof
in the Appendix.

The paper is organized as follows: The next section deals firstly with
the mathematical preliminaries, namely the introduction of the
weighted spaces. 
Then we present the main results, namely the existence and well
posedness together with the main properties of the solutions obtained. 
The properties of the weighted Sobolev spaces $H_{s,\delta}$ are presented
in Section \ref{sec:more-useful-tools}. 
For the proofs of those properties, we refer to \cite{BK3,
  BK7, 
  triebel76:_spaces_kudrj2, triebel76:_spaces_kudrj}, except Lemma
\ref{lem:3}, which is new and crucial for the proof of the nonlinear
power estimate, Proposition \ref{prop:power}.
In section \ref{sec:cauchy-problem} we establish the main mathematical
tools, including the local existence and well posedness of symmetric
hyperbolic systems in the $H_{s,\delta}$ weighted spaces, two energy type
estimates of the solutions to hyperbolic systems, the elliptic estimate
for the Poisson equation and two non--linear estimates.
The last section is dedicated to the proof of the main result using a
fixed--point argument. 
In the Appendix \ref{sec:proof-modif-banach} we present and prove a
modified version of the Banach fixed--point theorem. 
   
\section{Main results}
\label{sec:main-result}
We obtain well posedness of the Euler--Poisson--Makino system
(\ref{eq:iteration-scheme:1b})--(\ref{eq:iteration-scheme:3b}) for
densities without compact support but with a polynomial decay at
infinity, and with the equation of state (\ref{eq:state}).
The class of solutions we obtain have finite mass, a finite energy
functional, and moreover, they contain the static spherical static
symmetric solutions of for the adiabatic constant $\gamma=\frac{6}{5}$ (see
Subsection \ref{sec:properties-solutions}). 
These solutions are continuously differentiable and they are also
classical solutions of the Euler--Poisson system
(\ref{eq:iteration-scheme:1})--(\ref{eq:iteration-scheme:3}).

The Euler--Poisson--Makino system is considered in the weighted Sobolev spaces
of fractional order $H_{s,\delta}$.

So we first define these spaces.
Let $\{\psi_j\}_{j=0}^\infty$ dyadic partition of unity in $\setR^3$,
that is, $ \psi_j\in C_0^\infty(\setR^3)$, $\psi_j(x)\geq0$,
$\mbox{supp}(\psi_j)\subset \{x: 2^{j-2}\leq |x| \leq 2^{j+1}\}$,
$\psi_j(x)=1$ on $\{x: 2^{j-1}\leq |x| \leq 2^{j}\}$ for $j=1,2,...$,
$\mbox{supp}(\psi_0)\subset\{x:|x|\leq 2\}$, $\psi_0(x)=1$ on
$\{x: |x|\leq 1\}$ and
\begin{equation}
  \label{eq:weighted:1}
  |\partial^\alpha  \psi_j(x)|\leq  C_\alpha  2^{-|\alpha|j},
\end{equation}
where the constant $C_\alpha$ does not depend on $j$.
We denote by $H^s$ the Sobolev spaces with the norm given by
\begin{displaymath}
  \|u\|_{H^s}^2=\int(1+|\xi|^2)^s |\hat{u}(\xi)|^2d\xi,
\end{displaymath}
where $\hat{u}$ is the Fourier transform of $u$. 
The scaling by a positive number $\epsilon$ is denoted by
$f_\varepsilon(x)=f(\varepsilon x)$.
\begin{defn}[Weighted fractional Sobolev spaces]
  \label{def:weighted:3}
  Let $s,\delta\in\setR$, the weighted Sobolev space $H_{s,\delta}$ is
  the set of all tempered distributions such that the norm
  \begin{equation}
    \label{eq:weighted:4}
    \left(\|u\|_{H_{s,\delta}}\right)^2=
    \sum_{j=0}^\infty  2^{( \frac{3}{2} + \delta)2j} \| (\psi_j
    u)_{(2^j)}\|_{H^{s}}^{2}
  \end{equation}
  is finite.
\end{defn}
The largest integer less than or equal to $s$ is denoted by $[s]$.
In this setting our main result is the following. 
\begin{thm}[Well posedness of the Euler--Poisson--Makino system]
  \label{thm:main:1}
  Let $1<\gamma < \frac{5}{3}$,
  $-\frac{3}{2}+\frac{2}{\left[\frac{2}{\gamma-1}\right]-1}\leq
  \delta<-\frac{1}{2}$,
  $ \frac{5}{2}<s$ if $\frac{2}{\gamma-1}$ is an integer and
  $\frac{5}{2}<s<\frac{5}{2}+\frac{2}{\gamma-1}-\left[\frac{2}{\gamma-1}\right]$
  otherwise. 
  Suppose $(w_0, v^a_0)\in H_{s,\delta}$ and $w_0\geq 0$, then there
  exists a positive $T$ which depends on the $H_{s,\delta}$-norm of
  the initial data and there exists and a unique solution $(w,v^a)$ of
  the Euler--Poisson--Makino system
  (\ref{eq:iteration-scheme:1b})--(\ref{eq:iteration-scheme:3b}) such
  that
  \begin{equation*}
    (w,v^a)\in C\left([0,T], H_{s,\delta}\right)\cap C^1\left([0,T], 
      H_{s-1,\delta+1}\right)
  \end{equation*}
  and $0\leq w(t,\cdot)$ in $[0,T]$.
 
\end{thm}
As we have already pointed out this theorem does not include the
continuity of the flow map with respect to the initial data. 
Nevertheless, it has a series of noteworthy corollaries which
we list below:
\subsection{Properties of the solutions}
\label{sec:properties-solutions}
We start with static solutions of the Euler--Poisson system. 
Those solutions must be spherical symmetric (see for example
\cite{lichtenstein28:_eigen_gleic_teilc_geset
}) and they can be obtained by solving the Lane--Emden equation
\cite{CHA}.
The linear stability has been an open problem for a long time, so it
is interesting to see whether a class of solutions can be constructed
which include static solutions. 
To the best of our knowledge, this has not been achieved for solutions
with finite radius.
For $\gamma=\frac{6}{5}$ there is one parameter family (parameterized by
the central density) of solutions which have \textit{finite mass} but
\textit{infinite radius}, and it is given by
\begin{equation}
  \label{eq:gamblin:9}
  \rho(t,x)=\rho(|x|)=a^{\frac{5}{2}}\left(a^2+|x|^2\right)^{-\frac{5}{2}}  \sim 
 |x|^{-5},
\end{equation}
where $a$ is a positive constant
see~\cite{CHA
}. 
The corresponding solutions in the Makino variable are given by
\begin{equation}
  \label{eq:section-0003:1}
  w(x,t)=a^{\frac{1}{4}}\left(a^2+|x|^2\right)^{-\frac{1}{4}}\sim 
|x|^{-\frac{1}{2}}.
\end{equation}
Such static solutions are included in the class of solutions whose
existence is guaranteed by Theorem \ref{thm:main:1}, as it is stated
in the following corollary.
\begin{cor}[The static solutions of the Euler--Poisson system]
\label{cor:section-0003:1}
  Let $\gamma=\frac{6}{5}$, $-\frac{23}{18}<\delta<-1$ and
  $\frac{5}{2}<s$. 
  Then there exists a positive $T$ and a unique solution $(w,v^a)$ to
  the Euler--Poisson--Makino system
  (\ref{eq:iteration-scheme:1b})--(\ref{eq:iteration-scheme:3b}) such
  that
  \begin{equation*}
    (w,v^a)\in C\left([0,T], H_{s,\delta}\right)\cap C^1\left([0,T], 
      H_{s-1,\delta+1}\right),
  \end{equation*}
  and for which the initial data include the static solution
  $w_0(x)=\left(a^2+|x|^2\right)^{-\frac{1}{4}}$.
\end{cor}
\begin{proof}
  The proof is straightforward. 
  As discussed above for $\gamma=\frac{6}{5}$, $\rho$ is given by
  equation (\ref{eq:gamblin:9}), while $w$ is given by equation
  (\ref{eq:section-0003:1}).
  Note that $(a^2+|x|^2)^{-\frac{1}{4}}\in H_{s,\delta}$ if
  $\delta <-1$. 
  On the other hand the lower bound for $\delta$ in Theorem
  \ref{thm:main:1} for $\gamma=\frac{6}{5}$ gives us
  \begin{math}
    -\frac{3}{2}+\frac{2}{9}=-\frac{23}{18}<-1
  \end{math}.
\end{proof}
  
Note that the well posedness is obtained in the term of the Makino
variable. 
Nevertheless, setting
$\rho(t,x)=c_{K,\gamma}w^{\frac{2}{\gamma-1}}(t,x)$, 
$c_{K,\gamma}=\left(\frac{2\sqrt{K\gamma}}{\gamma-1}\right)^{\frac{-2}
{\gamma-1}}$,
we also get a classical solution to the Euler--Poisson system
(\ref{eq:iteration-scheme:1})--(\ref{eq:iteration-scheme:3}).
 
\begin{cor}[Local solutions of the original Euler--Poisson system]
  \label{cor:section-0003:3}
  Let $1<\gamma<\frac{5}{3}$,
  $-\frac{3}{2}+\frac{3}{\left[\frac{2}{\gamma-1}\right]}\leq
  \delta<-\frac{1}{2}$,
  $ \frac{5}{2}<s$ if $\frac{2}{\gamma-1}$ is an integer and
  $\frac{5}{2}<s<\frac{5}{2}+\frac{2}{\gamma-1}-\left[\frac{2}{\gamma-1}\right]$
  otherwise. 
  Suppose $(\rho^{\frac{2}{\gamma-1}}_0,v^a_0)\in H_{s,\delta}$. Then
  there exists a positive $T$ and a unique  $C^1$--solution $(\rho, v^a)$  to 
the Euler--Poisson system
  (\ref{eq:iteration-scheme:1})--(\ref{eq:iteration-scheme:3}) with
  the equation the equation of state (\ref{eq:state}) such that
  \begin{equation*}
    \left(\rho(t,\cdot),v^a(t,\cdot)\right)\in  L^\infty([0,T],H_{s,\delta}).
  \end{equation*}
\end{cor}
Please note that the initial data in Corollary
\ref{cor:section-0003:3} are given by the Makino variable $w$ and not
by the physical quantity $\rho$. 
It is  an open problem to solve the Euler--Poisson system entirely
in terms of $\rho$ for situations in which $\rho$ could be zero.

\begin{proof}[Proof of Corollary \ref{cor:section-0003:3}]
  Set $w_0=c_{k,\gamma}\rho_0^{\frac{\gamma-1}{2}}$, then Theorem
  \ref{thm:main:1} provides a unique solution
  $\left(w(t,\cdot),v^a(t,\cdot)\right)\in H_{s,\delta}$ with the
  corresponding initial data. 
  By Propositions \ref{prop:emb-cont} and \ref{Kateb},
  $\|\rho(t,\cdot)\|_{H_{s,\delta}}\leq C
  \|w(t,\cdot)\|_{H_{s,\delta}}$. Since $s>\frac{5}{2}$ and 
$\delta>-\frac{3}{2}$, then by the embedding, Propositions \ref{prop:emb-cont} 
(ii), yields that $(\rho,v^a)\in C^1$,
  and obviously they  satisfy
  (\ref{eq:iteration-scheme:1})--(\ref{eq:iteration-scheme:3}). 
\end{proof}
There exists a wide range of publication concerning the non-linear
stability of stationary solutions of the Euler-Poisson system relying
on the method of energy functionals,
see for example Rein~\cite{jang08:_nonlin_euler_poiss,
  rein03:_non
}. 
Having this context in mind we turn to the question of finite mass and
finite energy functional.
\begin{cor}[Finite mass and finite energy functional]
  \label{cor:section-0003:2}
 The solutions obtained by Theorem \ref{thm:main:1} 
have the
  properties that,
  \begin{enumerate}
    \item $\rho(t,\cdot)\in L^1(\setR^3)$, that is, they have finite
    mass.
    \item The energy functional
    \begin{equation}
      \label{eq:energy-functional}
      E=E(\rho,v^a):=\int\left(\frac{1}{2}\rho 
        |v^a|^2+\frac{K\rho^\gamma}{\gamma-1}\right) 
      dx-\frac{1}{2}\iint\frac{\rho(t,x)\rho(t,y)}{|x-y|}dx dy
    \end{equation}
    is well defined for those solutions. 
  \end{enumerate}
\end{cor}
\subsection{The advantages of the $H_{s,\delta}$ spaces}
\label{subsection:advantages}
In this section, we discuss the consequences of our main
result, Theorem \ref{thm:main:1} and possible applications, and compare
them with previous results obtained by other authors.
\begin{itemize}
  \item We recall that the Euler--Poisson system
  (\ref{eq:iteration-scheme:1})-(\ref{eq:iteration-scheme:3})
  degenerates when the density approaches to zero and the only known
  method to solve an initial value problem in this context is to
  regularize the Euler equations by introducing the Makino variable
  (\ref{eq:proto-euler-poisson-banach:1}).  
  All the previous local existence results
  \cite{makino86:_local_exist_theor_evolut_equat_gaseous_stars,
    gamblin93:_solut_euler
    , Bezard_93
  }, including the present paper, have used this technique. 
  Thus in order to include the spherical symmetric static solutions of
  the Lane--Emden equation for $\gamma=\frac{6}{5}$ in our class of
  solutions, it is necessary to express it in terms of the Makino
  variable $w$. 
  But from (\ref{eq:section-0003:1}) we see that this function does
  not belong to the Sobolev $H^s$ space. 
  \item To overcome the difficulty with the Makino variable
  Gamblin uses uniformly locally Sobolev spaces $H_{ul}^s$ spaces which
  were introduced by Kato. 
  However as it was pointed out by Majda \cite[Thm 2.1, p.
  5 0]{majda84:_compr_fluid_flow_system_conser
  }, that for first order symmetric hyperbolic systems with a given initial
  data $u_0\in H^s_{ul}$, $\frac{3}{2}+1<s$ the corresponding solutions
  belong only to
  $C \left([0,T];H_{loc}^s \right)\cap C^1 \left([0,T];H_{loc}^{s-1}
  \right)\cap L^{\infty}\left([0,T];H_{ul}^s \right)$.
  Furthermore, continuity in the $H_{ul}^s$ norm causes a loss of
  regularity \cite[Theorem
  2.4]{gamblin93:_solut_euler
  }.
  We prove well-posedness in the $H_{s,\delta}$ spaces, Theorem
  \ref{thm:existence-quasi}, and circumvent these weaknesses of the
  uniformly locally Sobolev spaces.
  \item Another benefit of the $H_{s,\delta}$ spaces concerns the treatment
  of the Poisson equations. 
  The Laplacian is a Fredholm operator in those spaces
  \cite{mcowen79:_behav_sobol_spaces, cantor75:_spaces_funct_condit_r}
  (see Subsection \ref{sec:elliptic-estimate}), and for certain values
  of $\delta$ is an isomorphism. 
  Thus with the aid of the nonlinear power estimate, Proposition \ref{prop:power},
  we are able to treat both the hyperbolic and the elliptic
  part in the same type of Sobolev spaces. 
  On the contrary, the $H_{ul}^s$ are not suited for the Poisson
  equation. 
  To circumvent this difficulty Gamblin demands that the initial
  density $\rho_0\in W^{1,p}$, $1\leq p<3$. 
  Therefore he has two types of initial data, namely,
  $\rho_0\in W^{1,p}$ and the Makino variable
  $\rho_0^{\frac{\gamma-1}{2}}\in H_{ul}^s$. 
  However his initial data for the velocity $v^a_0$ belongs to $ H^s_{ul}$. 
  Under these initial conditions Gamblin proved that for
  $\frac{7}{2}<s<\frac{2}{\gamma-1}$ the solutions are:
  \begin{equation*}
      (\rho,v^a)\in \cap_{i=1,2} C^i \left( [0,T^{*}];H^{s^{\prime}-i}_{ul}
      \right),
      \quad s^{\prime}<s,\quad
      \rho\in L^{\infty}\left([0,T];W^{1,p}\cap H_{ul}^{s_\epsilon} \right),
  \end{equation*}
  where $s_\epsilon=\min\{\frac{2}{\gamma-1}-\epsilon, s\}$ if
  $\frac{2}{\gamma-1}\not\in\setN$ and $s_\epsilon =s$ otherwise.
  Thus the density belongs to $W^{1,p}$ and falls off at infinity,
  while the velocity is in $H_{ul}^s$ and therefore does not tend to zero. 
  Such a class solutions, even if it contains spherical symmetric
  static solutions, do not model isolated bodies in an appropriate
  way.
  \item The uniform Sobolev spaces $H^s_{ul}$, that Gamblin  used in order
  to include the static solutions for $\gamma=\frac{6}{5}$, are not suited
  for the Einstein--Euler system in an asymptotically flat setting. 
  Recall that in these functional spaces the Einstein constraint
  equations cannot be solved, while they can be solved using the
  $H_{s,\delta}$ spaces.
  The last question is important if one considers the Euler--Poisson
  system as the Newtonian limit of the Einstein--Euler system.

  Oliynyk \cite{oliynyk07:_newton_limit_perfec_fluid}
  proved the Newtonian limit in an asymptotically flat setting.  
  He showed that solutions of the Einstein-Euler system converges to
  solutions of the Euler--Poisson system, under the restriction that
  the density has compact support. 
  In order to generalize his result to the case where the density only
  falls off in an appropriate way one needs a functional setting which
  is suited for both systems. 
  While the weighted fractional Sobolev spaces are known to be
  appropriate, there is no existence result known for the Einstein
  equations 
  (plus matter fields) in an asymptotically flat situation using the
  functional setting of  $H_{ul}^s$ spaces.

\end{itemize}  
  
\section{Weighted fractional Sobolev spaces}
\label{sec:more-useful-tools}
The weighted Sobolev spaces whose weights vary with the order of the
derivatives and which are of integer order can be defined as a
completion of $C_0^\infty(\setR^3)$ under the norm
\begin{equation}
  \label{eq:norm:1}
  \|u\|_{m,\delta}^2=\sum_{|\alpha|\leq 
    m}\|(1+|x|)^{\delta+|\alpha|}|\partial^\alpha u|\|_{L^2}^2.
\end{equation}
These spaces were introduced by Nirenberg and Walker
\cite{nirenberg73:_null_spaces_ellip_differ_operat_r}. 
Triebel extended them to fractional order and proved basic properties
such as duality, interpolation, and density of smooth functions
\cite{triebel76:_spaces_kudrj2}. 
Triebel expressed the fractional norm in an integral form and used the
dyadic decomposition of the norm (\ref{eq:weighted:4}) just in order
to derive certain properties. 
We have adopted it as a definition of the norm since it enables us to
extend many of the properties of Sobolev spaces $H^{s}$ to
$H_{s,\delta}$.  
\subsection{Properties of the weighted  fractional Sobolev spaces}
\label{sec:prop-weight-sobol}
Here we quote the propositions and properties which are needed for the
proof of the main result. 
For their proofs and further details see
\cite{BK3, BK7, triebel76:_spaces_kudrj2}.
\begin{thm}[Triebel, Basic properties]
  \label{thm:Triebel}
  Let $s,\delta\in \setR$.
  \begin{enumerate}
    \item[{\rm (a)}] The space $H_{s,\delta}$ is a Banach space and
    different choices of dyadic resolutions $\{\psi_j\}$ which
    satisfies (\ref{eq:weighted:1}) result in equivalent norms. 
    \item[{\rm (b)}] $C_0^\infty(\setR^3)$ is a dense subset in
    $H_{s,\delta}$.
    \item[{\rm (c)}] The topological dual space of $H_{s,\delta}$ is
    $H_{-s,-\delta}$.
    \item[{\rm (d)}]
    \label{thm:interpolation}
    Interpolation: Let $0<\theta<1$, $s=\theta s_0+(1-\theta)s_1$ and
    $\delta=\theta \delta_0+(1-\theta)\delta_1$, then
    \begin{math}
      \left[H_{s_1,\delta_1},
        H_{s_2,\delta_2}\right]_{\theta}=H_{s,\delta}.
    \end{math}
  \end{enumerate}
\end{thm}
It was shown in \cite[page 68]{BK7} that one can construct a dyadic
sequence $\{\psi_j\}$, namely the one which we have introduced in section
\ref{sec:main-result},  in such way that for any positive $\gamma$,
$\psi_j^\gamma \in C_0^\infty(\mathbb{R}^3)$, and for each multi-index
$\alpha$ there exist two constants $C_1(\gamma,\alpha)$ and $C_2(\gamma,\alpha)$ such that
\begin{equation*}
  C_1(\gamma,\alpha)|\partial^\alpha \psi_j(x)|\leq
  |\partial^\alpha \psi_j^\gamma(x)| 
  \leq C_2(\gamma,\alpha)|\partial^\alpha \psi_j(x)|.
\end{equation*}
These inequalities are independent of $j$. 
Hence $\{\psi_j^\gamma\}$ satisfies (\ref{eq:weighted:1}) and that is why  it  
is an admissible dyadic resolution and  therefore   by Theorem
\ref{thm:Triebel} (a), we obtain the following equivalence.
\begin{prop}
  \label{prop:equivalence:1}
  For any positive $\gamma$, the norm
  \begin{equation}
    \label{eq:norm:2}
    \|u\|_{H_{s,\delta,\gamma}}^2:=\sum_{j=0}^\infty 
    2^{(\delta+\frac{3}{2})2j}\left\|\left(\psi_j^\gamma 
        u\right)_{2^j}\right\|_{H^s}^2
  \end{equation}
  is equivalent to $\|u\|_{H_{s,\delta}}$.
\end{prop}
\begin{prop}[Triebel \cite{triebel76:_spaces_kudrj}]
  \label{equivalence} 
  Let $s=m$ be an integer and $\gamma$ be positive number, then the norms
  (\ref{eq:norm:1}) and (\ref{eq:norm:2}) are equivalent. 
  In particular
  \begin{equation}
    \label{eq:norm:3}
    \|u\|_{H_{0,\delta,\gamma}}^2\simeq \|u\|_{L^2_\delta}^2:=\int 
    (1+|x|)^{2\delta} |u(x)|^2d x.
  \end{equation}
\end{prop}
The monotonicity property presented below of the norm is a simple
consequence of the definition of the norm (\ref{eq:weighted:4}).
\begin{prop}
  \label{monoton}
  If $s_1\leq s_2$ and $\delta_1\leq\delta_2$, then
  \begin{math}
    \|u\|_{H_{s_1,\delta_1}}\leq \|u\|_{H_{s_2,\delta_2}}.
  \end{math}
\end{prop}
\begin{prop}
  \label{embedding1}
  If $u\in H_{s,\delta}$, then
  \begin{math}
    \|\partial_i u\|_{H_{s-1,\delta+1}}\leq \| u\|_{H_{s,\delta}}.
  \end{math}
\end{prop}
\begin{prop}[Multiplication]
  \label{Algebra}
  Let $s\leq \min\{s_1,s_2\}$, $s+\frac{3}{2}<s_1+s_2$, $0\leq s_1+s_2$ and
  $\delta-\frac{3}{2}\leq \delta_1+\delta_2$. 
  If $u\in H_{s_1,\delta_1}$ and $v\in H_{s_2,\delta_2}$, then
  \begin{equation*}
    \left\|uv\right\|_{H_{s,\delta}}\leq C 
    \left\|u\right\|_{H_{s_1,\delta_1}}
    \left\|v\right\|_{H_{s_2,\delta_2}},
  \end{equation*}
  and the positive constant $C$ is independent of the functions $u $ and $v$.
\end{prop}
We now present the Sobolev embedding theorem in the weighted spaces. 
For $\beta\in \mathbb{R}$, we denote by $L_\beta^\infty$ the set of
all functions such that the norm
\begin{equation*}
  \|u\|_{L^\infty_\beta}=\sup_{\mathbb{R}^3}\left((1+|x|)^{\beta} 
    |u(x)|\right)
\end{equation*}
is finite, and by $C_\beta^m$ the set of all functions having
continuous partial derivatives up to order $m$ and such that the norm
\begin{equation*}
  \|u\|_{C_\beta^m}=\sum_{|\alpha|\leq m}
  \sup_{\setR^3}\left((1+|x|)^{\beta+|\alpha|}|\partial^\alpha u(x)|\right)
\end{equation*}
is finite.
\begin{prop}[Sobolev embedding] 
  \label{prop:emb-cont}
\hfill
  \begin{enumerate}
    \item[(i)] If $\frac{3}{2}<s$ and $\beta\leq\delta+\frac{3}{2}$,
    then
    \begin{math}
      \|u\|_{L^\infty_\beta}\leq C \|u\|_{H_{s,\delta}}.
    \end{math}
    \item[(ii)] Let $m$ be a nonnegative integer, $m+\frac{3}{2}<s$
    and $\beta\leq\delta+\frac{3}{2}$, then
    \begin{equation*}
      \|u\|_{C^m_\beta}\leq C \|u\|_{H_{s,\delta}}.
    \end{equation*}
  \end{enumerate}
\end{prop}
\begin{prop}
  \label{prop: L-1}
  If $\frac{3}{2}<\delta$, then $L^1\subset L^2_\delta$.
\end{prop}
We prove this simple proposition since a proof is not found in the
standard literature such as \cite{BK3, BK7, triebel76:_spaces_kudrj2}.
\begin{proof}
  Since $(1+|x|)^{-\delta}\in L^2$ when $\frac{3}{2}<\delta$, we get
  by the Cauchy Schwarz inequality that
  \begin{equation*}
    \|u\|_{L^1}=\int (1+|x|)^{-\delta} (1+|x|)^{\delta}|u|dx\leq 
    \|(1+|x|)^{-\delta}\|_{L^2} \|u\|_{L^2_\delta}.
  \end{equation*}
\end{proof}
Next we present a Moser type estimate in the weighted spaces.
\begin{prop}
  \label{Moser}
  Let $F:\setR^m\to\setR^l$ be a $C^{N+1}$-function such that $F(0)=0$
  and where $N\geq [s]+1$. 
  Then there is a constant $C$ such that for any $u\in H_{s,\delta}$
  \begin{equation}
    \label{eq:14}
    \|F(u)\|_{H_{s,\delta}}\leq
    C\|F\|_{C^{N+1}}\left(1+\|u\|_{L^\infty}^N\right)
    \|u\|_{H_{s,\delta}}.
  \end{equation}
\end{prop}
The following Proposition was proved by Kateb in the $H^s$ spaces.
\begin{prop}
  \label{Kateb}
  Let $u\in H_{s,\delta}\cap L^\infty$, $1<\beta$,
  $ 0<s<\beta +\frac{1}{2}$ and $\delta\in \mathbb{R}$, then
  \begin{equation}
    \label{eq:kateb}
    \||u|^\beta\|_{H_{s,\delta}}\leq
    C(\|u\|_{L^\infty})
    \|u\|_{H_{s,\delta}}.
  \end{equation}
\end{prop}
Note that if $\frac{3}{2}<s$ and $-\frac{3}{2}\leq \delta$, then by
Proposition \ref{prop:emb-cont} the constants in the estimates
(\ref{eq:14}) and (\ref{eq:kateb}) are universal and do depend on the
$L^\infty$--norm.

\begin{prop}(An intermediate estimate)
  \label{prop:intermediate}
  Let $0<s<s^\prime$, then
  \begin{equation*}
    \| u\|_{H_{s,\delta}}\leq \| u\|_{H_{0,\delta}}^{1-\frac{s}{s^\prime}}\| 
    u\|_{H_{s^\prime,\delta}}^{\frac{s}{s^\prime}}.
  \end{equation*}
\end{prop}
We shall need the following approximation property. 
\begin{prop}
  \label{prop:approx}
  Let $s<s^\prime$ and $\epsilon$ be an arbitrary positive number. 
  Then for any $u\in H_{s,\delta}$, there is
  $u_\epsilon\in H_{s^\prime,\delta}$ such that
  \begin{equation*}
    \|u-u_\epsilon \|_{H_{s,\delta}}<\epsilon \quad \text{ and } \quad 
\|u_\epsilon 
    \|_{H_{s^\prime,\delta}}\leq C_\epsilon \|u \|_{H_{s,\delta}}.
  \end{equation*}
\end{prop}
\subsection{Estimates for products of functions}
\label{sec:estim-prod-funct}
We turn now to one the main ingredients of our proof which concerns
the estimates of products of functions.
Suppose $u_1,\ldots,u_m$ are functions in $H_{s,\delta}$, then
obviously the product $u=u_1\dots u_m$ has the same degree of
regularity provided that $\frac{3}{2}<s$. 
The question is whether the product has a better decay at infinity? 
That is, whether $u$ belongs to $H_{s,\delta'}$ for some
$\delta'>\delta$. 
The following Lemma gives a partial answer and plays a central role in the
proof of our main result.
\begin{lem}[Estimates for products of functions]
  \label{lem:3}
  Suppose $u_i\in H_{s,\delta_i}$ for $i=1,\ldots,m$, $\frac{3}{2}<s$
  and $\delta\leq \delta_1+\cdots+\delta_m+\frac{(m-1)3}{2}$, then
  $u=u_1u_2\cdots u_m\in H_{s,\delta}$ and
  \begin{equation*}
    \left\|u\right\|_{H_{s,\delta}}\leq C 
\prod_{i=1}^m\left\|u_i\right\|_{H_{s,\delta_i}}.
  \end{equation*}
\end{lem}
\begin{proof}
  An essential tool of the proof is Proposition \ref{prop:equivalence:1}
  that provides an equivalent norm. 
  We use the norm as given by (\ref{eq:norm:2}) with $\gamma=m$, then
  by the multiplication property in $H^s$, we obtain
  \begin{equation*}
    \begin{split}
      & \left\|u\right\|_{H_{s,\delta}}^2 \leq C
      \left\|u\right\|_{H_{s,\delta,m}}^2 \\ = C & \sum_j 2^{(
        \delta +\frac{3}{2})2j} \| \left(\psi_j^m \left(u_1u_2\cdots 
u_m\right)\right)_{(2^j)}\|_{H^{s}}^{2} \\
      \leq C & \sum_j 2^{(\delta +\frac{3}{2})2j} \|\left(\psi_j
        u_1\right)_{(2^j)}\|_{H^{s}}^{2}\cdots \|\left(\psi_j
        u_m\right)_{(2^j)}\|_{H^{s}}^{2}.
    \end{split}
  \end{equation*}
  Set $a_{i,j}=\|\left(\psi_j u_i\right)_{(2^j)}\|_{H^{s}}^{2}$, then
  by the assumption
  \begin{equation*}
    \left(\delta+\frac{3}{2}\right)\leq \sum_{i=0}^m \left(\delta_i+\frac{3}{2}\right),
  \end{equation*}
  Hölder's inequality, and the elementary inequality
  \begin{math}
    \left(\sum_j a_{ij}^m\right)^{1/m}\leq \sum_j a_{ij}
  \end{math}
  (see
  e.g.~\cite[\S1.4]{hardy34:_inequal
  }, we have that
  \begin{equation*}
    \begin{split}
      \left\|u\right\|_{H_{s,\delta}}^2 & \leq C
      \sum_{j=0}^\infty\left( 2^{(\delta_+\frac{3}{2})2j}\prod_{i=0}^m
        a_{i,j}\right) \leq C \sum_{j=0}^\infty\left( \prod_{i=0}^m
        2^{(\delta_i+\frac{3}{2})2j}a_{i,j}\right)\\ & \leq
      C\prod_{i=0}^m\left(\sum_{j=0}^\infty\left(
          2^{(\delta_i+\frac{3}{2})2j}a_{i,j}\right)^m\right)^{\frac{1}{m}}
      \leq C \prod_{i=0}^m\left(\sum_{j=0}^\infty\left(
          2^{(\delta_i+\frac{3}{2})2j}a_{i,j}\right)\right)\\ & = C
      \prod_{i=0}^m\left( \left\|u_i\right\|_{H_{s,\delta_i}}^2\right).
    \end{split}
  \end{equation*}
\end{proof}
\begin{cor}[Powers of functions]
  Suppose $u\in H_{s,\delta}$, $\frac{3}{2}<s$ and $m$ be an integer
  greater than one.
  Then $u^m\in H_{s,\delta+\delta_0}$ whenever
  $\frac{\delta_0}{m-1}-\frac{3}{2}\leq \delta$.
\end{cor}

\section{Mathematical tools}
\label{sec:cauchy-problem}
In this section, we establish the tools needed for the proof of the
main result. 
These comprise of the energy estimates and the local existence theorem for
quasilinear symmetric hyperbolic systems, the solution to the Poisson
equation, as well as elliptic estimates, and an estimate of the power of
functions, all these are dealt in the weighted Sobolev spaces. 
We shall use the notation $x\lesssim y$ to denote an inequality
$x\leq C y$, where the positive constant $C$ depends on the parameters in
question.
\subsection{Symmetric hyperbolic systems}
\label{sec:symm-hyperb-syst}
\begin{defn}[Symmetric hyperbolic systems]
  \label{def:proto-euler-poisson-banach:1}
  We call a system of the form
  \begin{equation*}
    A^0(U)\partial_t U +  A^a(U)\partial_a U + B(U)U=F(t,x)
  \end{equation*}
  a symmetric hyperbolic system under the following assumptions:
  \begin{enumerate}
    \item
    \label{item:wellposs-einstein-euler-hyper:1}
    $A^\alpha$ are symmetric matrices for $\alpha=0,1,2,3$;
    \item $A^0$ is uniformly positive definite;
    \item \label{item:wellposs-einstein-euler-hyper:2} $A^\alpha$ and
    $B$ are smooth.
  \end{enumerate}
\end{defn}
\begin{rem}
  \label{rem:proto-euler-poisson-banach:3}
  It is straightforward to check that the
  Euler--Poisson--Makino system
  (\ref{eq:iteration-scheme:1b})--(\ref{eq:iteration-scheme:2b}) in a
  matrix representation takes the form
  \begin{equation*}
    \begin{pmatrix}
      1 & 0
      \\
      0 & \delta_{ab}
    \end{pmatrix}
    \partial_t
    \begin{pmatrix}
      w
      \\
      v^b
    \end{pmatrix}
    +
    \begin{pmatrix}
      v^c & \frac{\gamma -1}{2}\delta^c_b
      \\
      \frac{\gamma -1}{2}\delta^c_a & \delta_{ab} v^c
    \end{pmatrix}
    \partial_c
    \begin{pmatrix}
      w
      \\
      v^b
    \end{pmatrix}
    =
    \begin{pmatrix}
      0\\
      -\partial_a\phi\\
    \end{pmatrix}.
  \end{equation*}
  Here $\delta^c_a$ denotes the Kronecker delta.
  This is obviously a symmetric hyperbolic system coupled with the
  Poisson equation. 
  Note that $A^0=\mathds{1}$.
\end{rem}
\subsection{The Cauchy problem, existence theorem and energy
  estimates}
\label{sec:cauchy-probl-exist}
We consider the Cauchy problem for quasilinear symmetric hyperbolic
systems of the form
\begin{equation}
  \label{eq:neu-existence:21}
  \begin{cases}
    & \displaystyle{\partial_t U + A^a(U)\partial_a U +
      B(U)U = F(t,x)}\\
    & U(0,x)=U_0(x)
  \end{cases},
\end{equation}
where $A^a(U)$ and $B(U)$ are $N\times N$ matrices such that
$A^a(0)=B(0)=0$, and $U$ and $F$ are vector valid functions in $\setR^N$.
The well-posedness of these systems in the Sobolev spaces $H^s$ is well known.
Here we establish it in the weighted spaces $H_{s,\delta}$.
\begin{thm}[Well posedness of first order hyperbolic symmetric systems
  in $H_{s,\delta}$]
  \label{thm:existence-quasi}
  Let $\frac{5}{2}<s$, $-\frac{3}{2}\leq\delta$, $U_0\in H_{s,\delta}$
  and $F(t,\cdot)\in C([0,T^0],H_{s,\delta})$ for some positive $T^0$.
  Then there exists a positive $T\leq T^0$ and a unique solution $U$
  to the system (\ref{eq:neu-existence:21}) such that
  \begin{equation*}
    U\in C([0,T],H_{s,\delta})\cap C^1([0,T],H_{s-1,\delta+1}).
  \end{equation*}
\end{thm}
\begin{rem}
  The conclusion of Theorem \ref{thm:existence-quasi} can be extended
  to a system with $A^0$ positive definite, and
  $A^0(u)-I\in H_{s,\delta}$ for $u\in H_{s,\delta}$, but since we do not use such a
  generalization we omit the details.
  
\end{rem}
An essential ingredient of the proof are the energy estimates for the
linearized system.
\subsubsection{Energy estimates in the $H_{s,\delta}$ spaces}
\label{sec:energy-estimates}
We consider the linearization of the system
(\ref{eq:neu-existence:21}):
\begin{equation}
  \label{eq:linear}
  \partial_t U +  A^a(t,x)
  \partial_a U +
  B(t,x)U = F(t,x),
\end{equation}
where the matrices $A^a$ are a symmetric.
In order to derive the energy estimates we introduce an inner--product
in the $H_{s,\delta}$ spaces.
So let
$\Lambda^s[U]=\mathcal{F}^{-1}\left((1+|\xi|^2)^{\frac{s}{2}}\mathcal{F}
(U)\right)$,
where $\mathcal{F}$ denotes the Fourier transform. 
Then
\begin{equation*}
  \langle U, V\rangle_s:=\langle  
\Lambda^s[U],\Lambda^s[V]\rangle_{L^2}=\int\left(\Lambda^s[U]\cdot\Lambda^s[V]
\right)dx
\end{equation*}
is an inner--product on $H^s$, here the dot $\cdot$ denotes the scalar product. 
Now we set
\begin{equation}
  \label{eq:inner}
  \langle U,V\rangle_{s,\delta}:=\sum_{j=0}^\infty 
  2^{(\delta+\frac{3}{2})2j}
  \left\langle\left(\psi_j^2U\right)_{2^j},\left(\psi_j^2V
    \right)_{2^j}\right\rangle_{s},
\end{equation}
then it is an inner-product on the $H_{s,\delta}$ spaces, and
by Proposition \ref{prop:equivalence:1} the norm
\begin{equation*}
  \langle U,U\rangle_{s,\delta}=\sum_{j=0}^\infty 
2^{(\delta+\frac{3}{2})2j}\|\left(\psi_j^2U\right)_{2^j}\|_{H^s}^2=\|U\|_{H_{s,
\delta,2}}^2
\end{equation*}
is equivalent to the norm (\ref{eq:weighted:4}).

\begin{lem}
  \label{lem:energy-existence:1}
  Suppose $\frac{5}{2}< s$, $-\frac{3}{2}\leq \delta$, $A^a$ are
  symmetric matrices and
  $A^a(t,\cdot), B(t,\cdot), F(t,\cdot)\in H_{s,\delta}$. 
  If $U(t)=U(t,\cdot)\in C^1([0,T], H_{s,\delta})$ is a solution to
  the linear system (\ref{eq:linear}) for some positive $T$, then for
  $t\in [0,T]$
  \begin{equation}
    \label{eq:energy-estimates}
    \frac{1}{2}\frac{d}{dt}\langle 
    U(t),U(t)\rangle_{s,\delta}
    \leq C 
    \left(\|U(t)\|_{H_{s,\delta}}^2+\|F(t,\cdot)\|_{H_{s,\delta}}^2
    \right),   
  \end{equation}
 where the constant $C$ depends on the $H_{s,\delta}$ norm of the
  matrices $A^a$ and $B$.
\end{lem}
\begin{proof}
  Since $U(t)\in C^1([0,T], H_{s,\delta}))$ and it satisfies
  (\ref{eq:linear}), we have that
  \begin{equation*}
    \begin{split}
      & \frac{d}{2dt} \langle U(t),U(t)\rangle_{s,\delta} =\langle
      U(t),\partial_t U(t)\rangle_{s,\delta}\\ =-&\langle U(t),A^a
      \partial_a U(t) \rangle_{s,\delta}-\langle U(t), BU(t)
      \rangle_{s,\delta}+\langle U(t),F \rangle_{s,\delta}.
    \end{split}
  \end{equation*}

  Using the multiplicity property of $H_{s,\delta}$, Proposition
  \ref{Algebra}, and the Cauchy--Schwarz inequality we obtain that
  \begin{equation}
    \label{eq:inner:3}
    |\langle U(t), BU(t) \rangle_{s,\delta}|\leq 
    \|U(t)\|_{H_{s,\delta}}\|BU(t)\|_{H_{s,\delta}}\lesssim 
    \|B\|_{H_{s,\delta}}\|U(t)\|_{H_{s,\delta}}^2.
  \end{equation}
  Similarly, the last term
  \begin{equation}
    \label{eq:inner:4}
    |\langle U(t),F \rangle_{s,\delta}|\leq 
    \|U(t)\|_{H_{s,\delta}}\|F\|_{H_{s,\delta}}\leq 
    \frac{1}{2}\left(\|U(t)\|_{H_{s,\delta}}^2+\|F\|_{H_{s,\delta}}^2\right).
  \end{equation}
  The crucial point is the estimate of the terms with the matrices
  $A^a$. 
  So for a fixed index $a$ we set
  \begin{equation*}
  \begin{split}
  E_a(j) & =
\left\langle\left[\left(\psi_j^2U(t)\right)_{2^j}\right],\left[
\left(\psi_j^2\left(A^a\partial_a
          U(t)\right)\right)_{2^j}\right]\right\rangle_{s} \\ & =
  \left\langle\Lambda^s\left[\left(\psi_j^2U(t)\right)_{2^j}\right],
    \Lambda^s\left[\left(\psi_j^2\left(A^a\partial_a
          U(t)\right)\right)_{2^j}\right]\right\rangle_{L^2}, 
  \end{split}
  \end{equation*}
  then by the definition of the inner--product in $H_{s,\delta}$
  (\ref{eq:inner}), we have to show  the inequality
  \begin{equation}
    \label{eq:ineq:4}
    \langle U(t),A^a
      \partial_a U(t) \rangle_{s,\delta}=
    \sum_{j=0}^\infty 2^{(\delta+\frac{3}{2})2j}|E_a(j)|\leq C
    \| U\|_{H_{s,\delta}}^2,
  \end{equation}
  where the constant $C$ depends on the $H_{s,\delta}$ norm of the
  matrices $A^a$.

  We shall obtain it by applying the techniques of integration by
  parts, which requires the commutation of $\psi_jA^a$ with
  $\Lambda^s$.
  To do this we set $ \Psi_m=(\sum_{j=0}^\infty \psi_j)^{-1}\psi_m$,
  then $ \sum_{m=0}^\infty \Psi_m(x)=1$, hence we can replace $1$ by
  the infinite sum and get that for each $j$,
  \begin{equation}
    \label{eq:inner:2}
    \begin{split}
      E_a(j)
      =&\left\langle\Lambda^s\left[\left(\psi_j^2U(t)\right)_{2^j}\right],
        \Lambda^s\left[\left(\psi_j^2\left(A^a\partial_a
              U(t)\right)\right)_{2^j}\right]\right\rangle_{L^2}\\
      = &
      \left\langle\Lambda^s\left[\left(\psi_j^2U(t)\right)_{2^j}\right],
        \Lambda^s\left[\left(\psi_j^2\left(\left( \sum_{m=0}^\infty
                \Psi_m\right)A^a\partial_a
              U(t)\right)\right)_{2^j}\right]\right\rangle_{L^2}\\ = &
      \sum_{m=0}^\infty
      \left\langle\Lambda^s\left[\left(\psi_j^2U(t)\right)_{2^j}\right],
        \Lambda^s\left[\left(\psi_j^2\left( \Psi_m A^a\partial_a
              U(t)\right)\right)_{2^j}\right]\right\rangle_{L^2}.
    \end{split}
  \end{equation}
  Note that $\psi_j\Psi_m\not\equiv 0$ only when $j-4\leq m\leq j+4$,
  therefore the series (\ref{eq:inner:2}) has a finite number of
  non-zero terms. 
  We now make the commutation
  \begin{equation}
    \label{eq:commutation}
    \begin{split}
      & \Lambda^s\left[\left(\psi_j^2\left(\Psi_m A^a\partial_a
            U(t)\right)\right)_{2^j}\right]\\ = &
      \Lambda^s\left[\left(\psi_j^2\left( \Psi_m A^a\partial_a
            U(t)\right)\right)_{2^j}\right] -\left( \Psi_m
        A^a\right)_{2^j}\Lambda^s\left[\left(\psi_j^2\partial_a
          U(t)\right)_{2^j}\right] \\ + & \left( \Psi_m
        A^a\right)_{2^j}\Lambda^s\left[\left(\psi_j^2\partial_a
          U(t)\right)_{2^j}\right].
    \end{split}
  \end{equation}
  Then we estimate the first term by the Kato--Ponce commutator inequality
  \cite[\S3.6]{kato88:_commut_euler_navier_stokes
  }, and get that
  \begin{equation}
    \label{eq:kp}
    \begin{split}
      &\left\|\Lambda^s\left[\left(\psi_j^2\left( \Psi_m A^a\partial_a
              U(t)\right)\right)_{2^j}\right] -\left( \Psi_m
          A^a\right)_{2^j}\Lambda^s\left[\left(\psi_j^2\partial_a
            U(t)\right)_{2^j}\right]\right\|_{L^2} \\ \lesssim &
      \left\|\nabla\left( \Psi_m
          A^a\right)_{2^j}\right\|_{L^\infty}\left\|
        \left(\psi_j^2\partial_a U(t)\right)_{2^j}\right\|_{H^{s-1}}+
      \left\|\left( \Psi_m A^a\right)_{2^j}\right\|_{H^s}\left\|
        \left(\psi_j^2\partial_a U(t)\right)_{2^j}\right\|_{L^\infty}.
    \end{split}
  \end{equation}
  For the second term of the left hand side of (\ref{eq:commutation})
  we use the symmetry of $A^a$ and then by integration by parts we
  obtain that
  \begin{equation}
    \label{eq:int-part}
    \begin{split}
      2\left\langle
        \Lambda^s\left[\left(\psi_j^2U(t)\right)_{2^j}\right],
        \left(\Psi_m
          A^a\right)_{2^j}\Lambda^s\left[\left(\psi_j^2\partial_a
            U(t)\right)_{2^j}\right]\right\rangle_{L^2}\\ = 4
      \left\langle
        \Lambda^s\left[\left(\partial_a\psi_j\psi_jU(t)\right)_{2^j}\right],
        \left(\Psi_m A^a\right)_{2^j}\Lambda^s\left[\left(\psi_j^2
            U(t)\right)_{2^j}\right]\right\rangle_{L^2} \\ +
      \left\langle
        \Lambda^s\left[\left(\psi_j^2U(t)\right)_{2^j}\right],
        \partial_a\left(\Psi_m
          A^a\right)_{2^j}\Lambda^s\left[\left(\psi_j^2
            U(t)\right)_{2^j}\right]\right\rangle_{L^2}.
    \end{split}
  \end{equation}
  Setting
  \begin{equation*}
    E_a(j,m)= 
    \left\langle
      \Lambda^s\left[\left(
          \psi_j^2U(t)\right)_{2^j}\right],\Lambda^s\left[\left(\psi_j^2\left(
            \Psi_m A^a\partial_a
            U(t)\right)
        \right)_{2^j}\right]\right\rangle_{L^2},
  \end{equation*}
  then by inequality (\ref{eq:kp}), equality (\ref{eq:int-part}) and
  the Cauchy Schwarz inequality, we obtain that
  \begin{equation}
    \label{eq:ineq:1}
    \begin{split}
      |E_a(j,m)| & \lesssim \left\|\nabla\left(\Psi_m
          A^a\right)_{2^j}\right\|_{L^\infty}\left\|\left(\psi_j^2
          U(t)\right)_{2^j}\right\|_{H^s}\left\|\left(\psi_j^2
          \partial_aU(t)\right)_{2^j}\right\|_{H^{s-1}} \\ & +
      \left\|\left(\Psi_m
          A^a\right)_{2^j}\right\|_{H^s}\left\|\left(\psi_j^2
          U(t)\right)_{2^j}\right\|_{H^s}\left\|\left(\psi_j^2
          \partial_aU(t)\right)_{2^j}\right\|_{L^\infty} \\ & +
      2\left\|\left(\Psi_m
          A^a\right)_{2^j}\right\|_{L^\infty}\left\|\left(\psi_j
          U(t)\right)_{2^j}\right\|_{H^s}\left\|\left(\psi_j^2
          U(t)\right)_{2^j}\right\|_{H^{s}}\\ & +\frac{1}{2}
      \left\|\partial_a\left( \Psi_m
          A^a\right)_{2^j}\right\|_{L^\infty}\left\|\left(\psi_j^2
          U(t)\right)_{2^j}\right\|_{H^s}^2.
    \end{split}
  \end{equation}
  Note that $\Psi_m(x)=f(x)\psi_m(x)$, where $f\in C^\infty$. 
  Hence
  $\|(\Psi_m A^a)_{2^j}\|_{H_{s,\delta}}\lesssim \|(\psi_m
  A^a)_{2^j}\|_{H_{s,\delta}}$
  Now, taking into account the Sobolev inequality, we have that
  \begin{equation*}
    \left\|\nabla\left( 
        \Psi_m A^a\right)_{2^j}\right\|_{L^\infty}\lesssim  \left\|
      \nabla(\Psi_m 
      A^a)_{2^j}\right\|_{H^{s-1}}\lesssim \left\|\left( 
        \psi_m A^a\right)_{2^j}\right\|_{H^{s}}
  \end{equation*}
  and
  \begin{equation}
    \label{eq:ineq:2}
    \left\| \left(\psi_j^2\partial_a 
        U(t)\right)_{2^j}\right\|_{L^\infty}\lesssim  \left\| 
\left(\psi_j^2\partial_a 
        U(t)\right)_{2^j}\right\|_{H^{s-1}}.
  \end{equation}
  We recall that $E_a(j,m)\neq 0$ only if $j-4\leq m\leq j+4$, hence
  by inequalities (\ref{eq:ineq:1})-(\ref{eq:ineq:2}) and equality
  (\ref{eq:commutation}) we obtain that
  \begin{equation}
    \label{eq:ineq:3}
    \begin{split}
      & \sum_{j=0}^\infty 2^{(\delta+\frac{3}{2})2j}|E_a(j)|=
      \sum_{j=0}^\infty\sum_{m=j-4}^{j+4}
      2^{(\delta+\frac{3}{2})2j}|E_a(j,m)|\\
      \lesssim & \sum_{j=0}^\infty\sum_{m=j-4}^{j+4}
      2^{(\delta+\frac{3}{2})2j}\left\|\left( \psi_m
          A^a\right)_{2^j}\right\|_{H^{s}} \left\|\left(
          \psi_j^2U(t)\right)_{2^j}\right\|_{H^{s}}\left\|\left(
          \psi_j^2\partial_a U(t)\right)_{2^j}\right\|_{H^{s-1}} \\ +
      & \sum_{j=0}^\infty\sum_{m=j-4}^{j+4}
      2^{(\delta+\frac{3}{2})2j}\left\|\left( \psi_m
          A^a\right)_{2^j}\right\|_{H^{s}} \left\|\left(
          \psi_jU(t)\right)_{2^j}\right\|_{H^{s}}\left\|\left(
          \psi_j^2U(t)\right)_{2^j}\right\|_{H^{s}} \\ + &
      \sum_{j=0}^\infty\sum_{m=j-4}^{j+4}
      2^{(\delta+\frac{3}{2})2j}\left\|\left( \psi_m
          A^a\right)_{2^j}\right\|_{H^{s}} \left\|\left(
          \psi_j^2U(t)\right)_{2^j}\right\|_{H^{s}}^2.
    \end{split}
  \end{equation}
  We estimate now the first term of the right hand side of
  (\ref{eq:ineq:3}). 
  Utilizing the Hölder inequality and the fact that
  $2(\delta+\frac{3}{2})\leq(\delta+\frac{3}{2})+ (\delta+\frac{3}{2})
  +(\delta+1+\frac{3}{2})$ for $-\frac{5}{2}\leq \delta$, we get that
  \begin{equation*}
    \begin{split}
      &\sum_{j=0}^\infty\sum_{m=j-4}^{j+4}
      2^{(\delta+\frac{3}{2})2j}\left\|(\psi_m
        A^a)_{2^j}\right\|_{H^{s}}\left\|(\psi_j^2
        U(t))_{2^j}\right\|_{H^{s}}\left\|(\psi_j^2\partial_a
        U(t))_{2^j}\right\|_{H^{s-1}} \\ \lesssim &
      \left(\sum_{j=0}^\infty\sum_{m=j-4}^{j+4}
        2^{(\delta+\frac{3}{2})2j}\left\|(\psi_m
          A^a)_{2^j}\right\|_{H^{s}}^2\right)^{\frac{1}{2}}
      \left(\sum_{j=0}^\infty\sum_{m=j-4}^{j+4}\left(
          2^{(\delta+\frac{3}{2})2j}\left\|(\psi_j^2
            U(t))_{2^j}\right\|_{H^{s}}^2\right)^2\right)^{\frac{1}{4}}
      \\
      \times & \left(\sum_{j=0}^\infty\sum_{m=j-4}^{j+4}\left(
          2^{(\delta+1+\frac{3}{2})2j}\left\|(\psi_j^2
            \partial_a
            U(t))_{2^j}\right\|_{H^{s-1}}^2\right)^2\right)^{\frac{1}{4}}.
    \end{split}
  \end{equation*}
  By the elementary inequality
  \begin{math}
    \left(\sum_j a_j^2\right)^{1/2}\leq \sum_j a_j
  \end{math}
  (see e.g.~\cite[\S1.4]{hardy34:_inequal
  }),
  \begin{equation*}
    \begin{split}
      &\left(\sum_{j=0}^\infty\sum_{m=j-4}^{j+4}\left(2^{(\delta+\frac{3}{2})2j}
          \left\|(\psi_j^2
            U(t))_{2^j}\right\|_{H^{s}}^2\right)^2\right)^{\frac{1}{4}}
      \\ \leq &
      \left(\sum_{j=0}^\infty\sum_{m=j-4}^{j+4}\left(2^{(\delta+\frac{3}{2})2j}
          \left\|(\psi_j^2
            U(t))_{2^j}\right\|_{H^{s}}^2\right)\right)^{\frac{1}{2}}
      \leq C \|U(t)\|_{H_{s,\delta}}.
    \end{split}
  \end{equation*}
  Likewise, the last term in the product is less than
  $ C \|\partial_a U(t)\|_{H_{s-1,\delta+1}} \leq C \|
  U(t)\|_{H_{s,\delta}}$.
  For the first term in the product we need to use scaling properties
  of the $H^s$--norm, that is,
  \begin{equation*}
    \left\|\left( \psi_m A^a\right)_{2^j}\right\|_{H^{s}}=
    \left\|\left( (\psi_m A^a)_{2^m}\right)_{2^{j-m}}\right\|_{H^{s}}
    C(2^{j-m}) \left\| (\psi_m A^a)_{2^m}\right\|_{H^{s}}.
  \end{equation*}
  Note that $2^{-4}\leq 2^{j-m}\leq 2^4$, hence $C(2^{j-m})$ is
  bounded be a constant that is independent of $m$ and $j$. 
  Hence
  \begin{equation*}
    \left(\sum_{j=0}^\infty\sum_{m=j-4}^{j+4}
      2^{(\delta+\frac{3}{2})2j}\left\|(\psi_m
        A^a)_{2^j}\right\|_{H^{s}}^2\right)^{\frac{1}{2}}\leq C \|
    A^a\|_{H_{s,\delta}}.
  \end{equation*}
  Thus
  \begin{equation*}
    \begin{split}
      & \sum_{j=0}^\infty\sum_{m=j-4}^{j+4}
      2^{(\delta+\frac{3}{2})2j}\left\|(\psi_m
        A^a)_{2^j}\right\|_{H^{s}}\left\|(\psi_j^2
        U(t))_{2^j}\right\|_{H^{s}}\left\|(\psi_j^2\partial_a
        U(t))_{2^j}\right\|_{H^{s-1}} \\ & \leq C
      \|A^a\|_{H_{s,\delta}} \|U(t)\|_{H_{s,\delta}}^2.
    \end{split}
  \end{equation*}
  
  In a similar manner we can estimate the second and the third term of
  the right hand side of (\ref{eq:ineq:3}) and get inequality
  (\ref{eq:ineq:4}).  
  Adding this inequality to (\ref{eq:inner:3}), (\ref{eq:inner:4}) we
  obtain inequality (\ref{eq:energy-estimates}) and that completes the
  proof.
\end{proof}
Energy estimates in a lower norm are needed for the contraction. 
We denote by $L^2_\delta$ the $L^2$ space with the weight
$(1+|x|)^\delta$. 
Obviously $H_{0,\delta}\simeq L_\delta^2$ (see Proposition
\ref{equivalence}).
\begin{lem}
  \label{lem:energy-L:1}
  Let $U(t)\in C^1([0,T], L_\delta^2)$ be a solution to the linear
  system (\ref{eq:linear}) for some positive $T$, then
  \begin{equation}
    \label{eq:energy-estimates:L2}
    \frac{1}{2}\frac{d}{dt}\|U(t)\|^2_{L_\delta^2}\leq
    C 
    \left(\| U(t)\|_{L_{\delta}}^2+ \| F(t,\cdot)\|_{L_\delta}^2\right), \quad 
    t\in[0,T] 
  \end{equation}
  and the constant $C$ depends on the $L^\infty$ norm of $A^a$,
  $\partial_a A^a$ and $B$.
\end{lem}
\noindent
\textit{The idea of the proof.} 
Since $U(t)$ is a solution to (\ref{eq:linear}),
\begin{equation*}
  \begin{split}
    &\frac{1}{2}\frac{d}{dt}\|U(t)\|^2_{L_\delta^2}=\int
    (1+|x|)^{2\delta} \left(U(t)\cdot \partial_t U(t)\right)dx \\ = 
    - & \sum_{a=1}^3\int (1+|x|)^{2\delta} \left(U(t)\cdot A^a \partial_a
      U(t)\right)dx-\int (1+|x|)^{2\delta} \left(U(t)\cdot B
      U(t)\right)dx \\ + &\int (1+|x|)^{2\delta} \left(U(t)\cdot
      F\right) dx.
  \end{split}
\end{equation*}
Applying integration by parts, the Cauchy--Schwarz inequality and
$L^\infty$--$L^2$ estimates we obtain (\ref{eq:energy-estimates:L2}).

\subsubsection{Proof of Theorem \ref{thm:existence-quasi}}
\label{sec:proof-theor-refthm}
\begin{proof}
  We are using the  known iteration scheme
  \cite{majda84:_compr_fluid_flow_system_conser
  }. 
  In order to do that we need to approximate the initial data and the
  right hand side of (\ref{eq:neu-existence:21}) by smooth functions. 
  Since $C_0^\infty$ is dense in $H_{s,\delta}$ (see Theorem
  \ref{thm:Triebel} (b) and Proposition \ref{prop:approx}), there are
  two sequences
  $\{U_0^k\}, \{F^k(t,\cdot)\}\subset C_0^\infty(\setR^3)$ such that
  \begin{eqnarray}
    \label{eq:iteration:5}
    \| U_0^0\|_{H_{s+1,\delta}}\leq C_0 \|U_0\|_{H_{s,\delta}},\\
    \label{eq:iteration:6}
    \| U_0^k-U_0\|_{H_{s,\delta}}^2\leq 2^{-k},\\ 
    \label{eq:iteration:7}
    \sup_{0\leq t\leq T^0}\|F^k(t,\cdot)-F(t,\cdot)\|_{H_{s,\delta}}^2\leq 
    2^{-k}.
  \end{eqnarray}
  We set now $U^0(t,x)=U^0_0(x)$ and let $U^{k+1}(t,x)$ be the
  solution to the linear initial value problem
  \begin{equation}
    \label{eq:iteration:1}
    \begin{cases}
      & \displaystyle{\partial_t U^{k+1} + A^a(U^k)\partial_a U^{k+1}
        +
        B(U^k)U^{k+1} = F^k}\\
      & U(0,x)=U_0^k(x)
    \end{cases}.
  \end{equation}
  Since the linear system (\ref{eq:iteration:1}) has $C_0^\infty$
  coefficients, $\{U^k(t,\cdot)\}\subset C_0^\infty(\setR^3)$ (see
  e.g.~\cite{john86:_partial
  }). 
  Hence for each positive $R$ and integer $k$,
  \begin{equation}
    \label{eq:iteration:2}
    T_k:=\sup\{T:\sup_{0\leq t\leq T} \|U^k(t)-U_0^0\|_{H_{s,\delta}}^2\leq 
R^2\}
  \end{equation}
  is finite.  
  
  We now choose $R $ so that
  $(8C_0^2 \|U_0\|_{H_{s,\delta}}^2+2)\leq R^2 $ and prove by
  induction that there is $0<T^*$ such that $T^{*}\leq T_k$ for all
  $k\geq 1$. 
  Set $V^{k+1}=U^{k+1}-U_0^0$, then it satisfies the linear system
  \begin{equation}
    \label{eq:iteration}
    \begin{split}
      \partial_t V^{k+1} + A^a(U^k)\partial_a V^{k+1} + B(U^k)V^{k+1}
      =F^k+ A^a(U^k)\partial_a U_0^0 + B(U^k)U_0^{0},
    \end{split}
  \end{equation}
  with
  \begin{equation*}
    V^{k+1}(0,x)=U_0^k(x)-U_0^0(x).
  \end{equation*}
  We apply Moser type estimates in the $H_{s,\delta}$ spaces,
  Propositions \ref{Moser} and \ref{Algebra}, (\ref{eq:iteration:5})
  and (\ref{eq:iteration:2}), then we conclude that there is a
  positive constant $C_1=C_1(R,\|U_0\|_{H_{s,\delta}})$ such that
  $\|A^a(U^k)\|_{H_{s,\delta}}^2\leq C_1$. 
  Similarly the other terms of (\ref{eq:iteration}) can be bounded by the
  same constant. 
  Applying Lemma \ref{lem:energy-existence:1}, we obtain that
  \begin{equation*}
    \frac{d}{dt}\|V^{k+1}(t)\|_{H_{s,\delta}}^2\leq 
    C_1\left(\|V^{k+1}(t)\|_{H_{s,\delta}}^2+ 
      \|F^{k}(t,\cdot)\|_{H_{s,\delta}}^2\right).
  \end{equation*}
  Then by the Gronwall inequality,
  (\ref{eq:iteration:5})-(\ref{eq:iteration:7}) and
  (\ref{eq:iteration}) we have that
  \begin{equation}
    \label{eq:Gronwall}
    \begin{split}
      & \|V^{k+1}(t)\|_{H_{s,\delta}}^2 \leq
      e^{C_1t}\left(\|U^{k+1}_0-U_0^0\|_{H_{s,\delta}}^2
        +\int_0^t\|F^{k}(t,\cdot)\|_{H_{s,\delta}}^2 d\tau\right)\\
      \leq & e^{C_1t}\left( 2^{-k}+4C_0^2\|U_0\|_{H_{s,\delta}}^2
        +2\int_0^t
        \|F(\tau,\cdot)\|_{H_{s,\delta}}^2d\tau+t2^{-k+1}\right).
    \end{split}
  \end{equation}
  If $t=0$, then the right hand side of (\ref{eq:Gronwall}) is equal
  to $2^{-k}+C_0^2\|U_0\|_{H_{s,\delta}}^2$. 
  Since we have chosen $ (8C_0^2\|U_0\|_{H_{s,\delta}}^2+2)\leq R^2$, there is
  a positive $T^\ast$ such that the right hand side of
  (\ref{eq:Gronwall}) is less than $R^2$.
and hence
  \begin{equation}
  \label{eq:iteration:4}
   \sup_{0\leq t\leq T^*}\|U^k(t) -U_0^0\|_{H_{s,\delta}}^2\leq R^2.
  \end{equation}
Consequently the sequence $\{U^k\}$ is bounded in the $H_{s,\delta}$ norm.

From equation (\ref{eq:iteration:1}), and by the multiplication
estimates in the $H_{s,\delta}$ spaces, we have that
    \begin{equation*}
  \begin{split}
    \|\partial_t U^{k+1}(t)\|_{H_{s-1,\delta+1}} & \leq C\left(\sum_{a=1}^3
       \|\partial_a
      U^{k+1}\|_{H_{s-1,\delta+1}}\| A^a(U^{k})\|_{H_{s,\delta}}+\|
U^{k+1}\|_{H_{s,\delta}}\| B(U^{k})\|_{H_{s,\delta}}\right)\\ & +
     \|F^{k}(t,\cdot)\|_{H_{s,\delta}}.
   \end{split}
 \end{equation*}
 By the Moser type estimate, Proposition \ref{Moser}, the uniform
 bound (\ref{eq:iteration:4}) and the above estimate, we see that
 there is a constant $L$ independent of $k$ such that
   \begin{equation}
   \label{eq:Lip}
    \sup_{0\leq t\leq T^\ast}\|\partial_t U^{k}(t)\|_{H_{s-1,\delta+1}}\leq L.
   \end{equation}
  We show now the contraction in the $L_\delta^2$--norm. 
  More precisely, we claim that there are positive constants
  $\Lambda<1$, $T^{\ast\ast}\leq T^\ast$ and a converging sequence
  $\{\beta_k\}$ such that
  \begin{equation}
    \label{eq:contraction}
    \sup_{0\leq t\leq T^{\ast\ast}}\|U^{k+1}(t)-U^{k}(t)\|_{L_\delta^2}^2\leq 
    \Lambda
    \sup_{0\leq t\leq T^{\ast\ast}}\|U^{k}(t)-U^{k-1}(t)\|_{L_\delta^2}^2+ 
    \sum_{k}\beta_k.
  \end{equation}
  The difference $\left(U^{k+1}-U^{k}\right)$ satisfies the linear
  system
  \begin{equation*}
    \partial _t\left( U^{k+1}-U^{k}\right)+A^a(U^k)\partial_a 
    \left( U^{k+1}-U^{k}\right)+B(U^k)\left( 
      U^{k+1}-U^{k}\right)=\widetilde{F}^k,
  \end{equation*}
  where
  \begin{equation*}
    \widetilde{F}^k=- \left[ 
      A^a(U^{k})-A^a(U^{k-1})\right]\partial_a U^k-
    \left[B(U^{k})-B(U^{k-1}\right] U^k +F^{k}-F^{k-1}.
  \end{equation*}
  In order to apply the $L^2_\delta$--energy estimate, Lemma
  \ref{lem:energy-L:1}, we need to show that
  $\|A^a(U^k)\|_{L^\infty}$, $\|\partial_a A^a(U^k)\|_{L^\infty}$ and
  $\|B(U^k)\|_{L^\infty}$ are bounded by a constant that is
  independent of $k$ and to estimate the $L^2_\delta$ norm of
  $\widetilde{F}^k$. 
  Considering for example $\partial_a A^a(U^k)$ for a fixed index $a$,
  then by the weighted Sobolev inequality, Proposition
  \ref{prop:emb-cont} and Proposition \ref{embedding1},
  \begin{equation*}
    \|\partial_a A^a(U^k)\|_{L^\infty}\leq C \|\partial_a 
    A^a(U^k)\|_{H_{s-1,\delta+1}}\leq C \| A^a(U^k)\|_{H_{s,\delta}}
  \end{equation*}
  holds when $\frac{5}{2}<s$ and $-\frac{3}{2}\leq \delta+1$. 
  In the previous step we showed that
  $ \| A^a(U^k)\|_{H_{s,\delta}}\leq C_1$.
  So we conclude that there is a constant
  $C_2=C_2(R,\|U_0\|_{H_{s,\delta}})$ such that
  \begin{equation*}
   \max\{\|A^a(U^k)\|_{L^\infty}, \|\partial_a A^a(U^k)\|_{L^\infty},
  \|B(U^k)\|_{L^\infty}\}\leq C_2.
  \end{equation*}
  Applying standard difference estimates we obtain that
  \begin{equation*}
    \|\left[A^a(U^k)-A^a(U^{k-1})\right]\partial_a U^k\|_{L^2_\delta}^2\leq 
    \|\partial_a U^{k}\|_{L^\infty}^2\sup\{|\nabla 
    A^a(U)|^2\}\|U^{k}-U^{k-1}\|_{L^2_\delta}^2,
  \end{equation*}
  where the supremum is taken over a ball with a radius that depends
  on $R$ and the initial data. 
  Taking into account (\ref{eq:iteration:5}), (\ref{eq:iteration:7})
  and (\ref{eq:iteration:2}), we see that there is a constant
  $C_3=C_3(R,\|U_0\|_{H_{s,\delta}})$ such that
  \begin{equation*}
    \|\widetilde{F}^k\|_{L^2_\delta}^2\leq C_3 
    \|U^{k}-U^{k-1}\|_{L^2_\delta}^2+6(2^{-k}).
  \end{equation*}
  Hence, by Lemma \ref{lem:energy-L:1}, the Gronwall inequality and
  (\ref{eq:iteration:6}) we obtain that
  \begin{equation}
    \label{eq:L2-Gronwall}
    \begin{split}
      & \|U^{k+1}(t)-U^{k}(t)\|_{L_\delta^2}^2\\ \leq e^{C_2t} &
      \left( \|U_0^{k+1}-U_0^k\|_{L^2_\delta}^2+ C_3\int_0^t
        \|U^{k}(\tau)-U^{k-1}(\tau)\|_{L_\delta^2}^2d\tau+
        t6(2^{-k})\right).
    \end{split}
  \end{equation}
  So we can choose $T^{\ast\ast}$ such that
  $e^{C_2T^{\ast\ast}}C_3T^{\ast\ast}=:\Lambda<1$ and  we set
  \begin{equation*}
   \beta_k=e^{C_2T^{\ast\ast}}(\left(\|U_0^{k+1}-U_0^k\|_{L^2_\delta}^2+T^{
      \ast\ast}6(2^{-k})\right).
  \end{equation*}
Since $\|U_0^{k+1}-U_0^k\|_{L^2_\delta}\lesssim \|U_0^{k+1}-U_0\|_{H_{s,\delta}}
+\|U_0^{k}-U_0\|_{H_{s,\delta}}$, the series $\sum_k \beta_k $ converges by   
(\ref{eq:iteration:6}).

  Having proved (\ref{eq:contraction}), we conclude that $\{U^k\}$ is a
  Cauchy sequence in $L^2_\delta$, and by the intermediate estimate,
  Proposition \ref{prop:intermediate} and the bound
  (\ref{eq:iteration:4}), it is also a Cauchy sequence in
  $H_{s',\delta}$ for any $0<s'<s$.

  Hence $U^k$ converges to $U\in H_{s',\delta}$, and if in addition
  $\frac{5}{2}<s'<s$, then by Sobolev embedding to the continuous,
  Proposition \ref{prop:emb-cont},
  $U\in C^1([0,T^{\ast\ast}],C(\setR^3))$ and it is a classical solution of
  (\ref{eq:neu-existence:21}).
 
  Following Majda
  \cite[Ch.~2]{majda84:_compr_fluid_flow_system_conser
  }, we show the weak limit
\begin{equation}
   \label{eq:weak}
   \lim_{k}\left\langle U^k,\varphi\right\rangle_{s,\delta}=\left\langle 
U,\varphi\right\rangle_{s,\delta} \quad \text{for all}\ \ \varphi\in 
H_{s,\delta}.
\end{equation} 
 Hence $\|U\|_{H_{s,\delta}}\leq \liminf_k \|U^k\|_{H_{s,\delta}}$ and 
consequently $U\in H_{s,\delta}$. 
To prove (\ref{eq:weak}) we take $s''>s $, arbitrary $\epsilon >0$, and by  
Proposition \ref{prop:approx} there is $\widetilde{\varphi}\in H_{s'',\delta}$ 
so that 
\begin{equation*}
 \left\|\varphi-\widetilde{\varphi}\right\|_{H_{s'',\delta}}<\epsilon\quad 
\text{and } \quad \left\|\widetilde{\varphi}\right\|_{H_{s'',\delta}}\leq 
C(\epsilon)\left\|{\varphi}\right\|_{H_{s,\delta}}.
\end{equation*}
Writing
\begin{equation*}
 \left\langle U^k- U,\varphi\right\rangle_{s,\delta}=\left\langle 
U^k-U,\widetilde{\varphi}\right\rangle_{s,\delta}+\left\langle 
U^k-U,\varphi-\widetilde{\varphi}\right\rangle_{s,\delta},
\end{equation*}
then 
\begin{equation*}
 |\left\langle 
U^k-U,\widetilde{\varphi}\right\rangle_{s,\delta}\leq 
\left\|U^k-U\right\|_{H_{s',\delta}}\left\|\widetilde{\varphi}\right\|_{H_{s'',
\delta}}\leq \left\|U^k-U\right\|_{H_{s',\delta}}
C(\epsilon)\left\|{\varphi}\right\|_{H_{s,\delta}}\to 0
\end{equation*}
as $k $ tends to infinity. And by (\ref{eq:iteration:2})
\begin{equation*}
|\left\langle 
U^k-U,\varphi-\widetilde{\varphi}\right\rangle_{s,\delta}|
\leq \left\|U^k-U\right\|_{H_{s,\delta}}
\left\|{\varphi}-\widetilde{\varphi}\right\|_{H_{s,\delta}}
\leq 
\sqrt{2}R\epsilon.
\end{equation*}

  Thus we have  shown the existence of a continuously differentiable solution 
$U$ to (\ref{eq:neu-existence:21}), which by  (\ref{eq:iteration:2}) and 
(\ref{eq:Lip}) belongs to 
  $L^\infty\left([0,T^{\ast\ast}],H_{s,\delta}\right)\cap {\rm
    Lip}\left([0,T^{\ast\ast}],H_{s-1,\delta+1}\right)$
  and  continuous with respect to the weak
  topology. It remains to prove uniqueness and well posedness. 
  The uniqueness is achieved by applying the $L^2_\delta$ energy
  estimates to the difference of two solutions. 
  Since $H_{s,\delta}$ are Hilbert spaces, it suffices to show that
  $\limsup_{t\to 0^+}\|U(t)\|_{H_{s,\delta}}\leq
  \|U_0\|_{H_{s,\delta}}$ in order to establish the well posedness. 
  We refer to
  \cite[Ch.~2]{majda84:_compr_fluid_flow_system_conser
  } and
  \cite[\S5]{karp11
  } for further details.
  This complete the proof of Theorem \ref{thm:existence-quasi}
\end{proof}

Suppose $U$ is a solution to (\ref{eq:neu-existence:21}), then it
follows from the proof of  Theorem \ref {thm:existence-quasi} that
$\{U(t):t\in[0,T]\}$ is contained in a compact set of $\setR^N$. 
Hence, by applying similar arguments as in the proof of Gronwall
inequality (\ref{eq:Gronwall}) to $U(t)$, we obtain the following
Corollary:
\begin{cor}
  \label{cor:1}
  Let $\frac{5}{2}<s$, $-\frac{3}{2}\leq\delta$ and assume that
  $U\in C\left([0,T],H_{s,\delta}\right)\cap C^1\left([0,T],H_{s-1,\delta+1}\right)$ is
  a solution to the Cauchy problem (\ref{eq:neu-existence:21}) such
  that $\|U_0\|_{H_{s,\delta}}\leq M_0$. 
  Then there is a positive constant $C_1$ that depends on $M_0$ such
  that
  \begin{equation}
    \label{eq:energy:1}
    \|U(t)\|_{H_{s,\delta}}^2\leq e^{C_1t}\left(M_0^2+\int_0^t 
      \|F(\tau,\cdot)\|_{H_{s,\delta}}^2d\tau\right).
  \end{equation}
\end{cor}
Likewise, if $U_i$ is a solution to
\begin{equation}
\label{eq:neu-existence:22}
  \begin{cases}
    & \displaystyle{\partial_t U_i + A^a(U_i)\partial_a U_i +
      B(U_i)U_i = F_i}\\
    & U_i(0,x)=U_0(x)
  \end{cases},\qquad i=1,2,
\end{equation}
then in a similar manner as in the proof of the $L^2_\delta$ Gronwall
inequality (\ref{eq:L2-Gronwall}) we get:
\begin{cor}
  \label{cor:2}
  Let $\frac{5}{2}<s$, $-\frac{3}{2}\leq\delta$ and suppose that
  $U_1,U_2\in C\left([0,T],H_{s,\delta}\right)\cap
  C^1\left([0,T],H_{s-1,\delta+1}\right)$
  are solutions to the Cauchy Problem (\ref{eq:neu-existence:22}) with
  the same initial data. 
  Then there is positive constants $C_2 $ such that
  \begin{equation}
    \label{eq:lower}
    \|(U_1-U_2)(t)\|_{L^2_\delta}^2\leq e^{C_2t}\int_0^t 
    \|(F_1-F_2)(\tau)\|_{L^2_\delta}^2d\tau.
  \end{equation}
\end{cor}
\subsection{The elliptic estimate}
\label{sec:elliptic-estimate}
We turn now to the solution of Poisson equation
\begin{equation}
  \label{eq:Poisson:2}
  \Delta\phi =4\pi\rho,
\end{equation}
which is coupled to the Euler--Poisson system. 
Since we consider a density function $\rho$ which may not have compact
support but could fall off at infinity, the ordinary Sobolev spaces
$H^{s}$ in $\setR^n$ are not an appropriate choice.
We chose to use weighted fractional Sobolev spaces, in which the
Laplace operator is invertible, and which are the only known spaces to
solve the Einstein--Euler system in this setting and hence could be
used to study the Newtonian limit.
Nirenberg and Walker initiated the study of elliptic equations in the
$H_{m,\delta}$ spaces of integer order
\cite{nirenberg73:_null_spaces_ellip_differ_operat_r}. 
Cantor \cite{cantor75:_spaces_funct_condit_r} proved that
\begin{equation}
  \label{eq:Delta}
  \Delta: H_{m,\delta}\to H_{m-2,\delta+2} 
\end{equation}
is an isomorphism in $\setR^3$ if $m$ is an integer and
$-\frac{3}{2}<\delta<-\frac{1}{2}$. 
McOwen showed that the operator $\Delta$ is a Fredholm operator if
$m=2$ and $\delta\neq -\frac{1}{2}+k, k\in \setZ$,
\cite{mcowen79:_behav_sobol_spaces}. 
Choquet--Bruhat and Christodoulou also proved the isomorphism of
(\ref{eq:Delta}) for $-\frac{3}{2}<\delta<-\frac{1}{2}$  in the weighted spaces 
of integer order
\cite{choquet--bruhat81:_ellip_system_h_spaces_manif_euclid_infin}. 
Using interpolation property of the $H_{s,\delta}$, Theorem
\ref{thm:Triebel} (d), we obtain:
\begin{thm}[(Cantor) Isomorphism of the Laplace operator]
  \label{thr:proto-euler-poisson-banach:1}
  Let $2\leq s $ be any real number and
  $\delta\in (-\frac{3}{2},-\frac{1}{2})$, then
  \begin{equation*}
    \Delta: H_{s,\delta}\to H_{s-2,\delta+2} 
  \end{equation*}
  is isomorphism. 
  Moreover, there is a constant $C$ such that
  \begin{equation*}
    \| u\|_{H_{s,\delta}}\leq C \|\Delta u\|_{H_{s-2,\delta+2}}\qquad \text{for 
all 
    }\ 
    u\in H_{s,\delta}.
  \end{equation*}
\end{thm}
Recall that equation (\ref{eq:iteration-scheme:2}) actually contains
the gradient of the solution of the Poisson equations. 
By the embedding (\ref{embedding1}), there is a constant $C_s$ such
that $\|\nabla u\|_{H_{s-1,\delta+1}}\leq C_s \| u\|_{H_{s,\delta}}$. 
So we conclude that there is a constant $C_e$ such that for any solution $\phi$ 
to the Poisson equation
(\ref{eq:Poisson:2}) satisfies the inequality
\begin{equation}
  \label{eq:Poisson:3}
  \|\nabla\phi\|_{H_{s-1,\delta+1}}\leq  C_e \|\rho\|_{H_{s-2,\delta+2}}.
\end{equation}
\subsection{The nonlinear power estimate}
\label{sec:nonlinear-estimate}
We turn now to nonlinear estimates of powers $u^\beta$ in the
$H_{s,\delta}$ spaces. 
Such type of estimates appears in several stages of the proofs, as well
as difference estimates in the $L^2_\delta$ spaces. 
Note that the symmetric hyperbolic system is considered in the
$H_{s,\delta}$ spaces with the weight $-\frac{3}{2}\leq\delta$. 
However for the Poisson equation, the source term $\rho$ needs to be in
$H_{s-1,\delta+2}$ and so that the weight $\delta$ has to be in the range of the
isomorphism of the Laplace operator, that is,
$\delta\in (-\frac{3}{2},-\frac{1}{2})$. 
Recall that the density is expressed by the Makino variable as
follows:
\begin{equation}
  \label{eq:non:6}
  \rho = c_{K\gamma}w^{\frac{2}{\gamma-1}},\qquad 
c_{K,\gamma}=\left(\frac{2\sqrt{K\gamma}}{\gamma-1}\right)^{\frac{-2}{\gamma-1}}
.
\end{equation}
Let us denote $\frac{\gamma-1}{2}$ by $\beta$, now given a nonnegative function
$w \in H_{s,\delta}$, we have to prove that
$w^\beta\in H_{s-1,\delta+2}$ for some
$\delta\in (-\frac{3}{2},-\frac{1}{2})$. 
The main tool of the proof is Lemma \ref{lem:3}.
\begin{prop}[Nonlinear estimate of power of functions]
  \label{prop:power}
  Suppose that $ w\in H_{s,\delta}$, $0\leq w$ and $\beta$ is a real number 
greater or
  equal $2$. 
  Then
  \begin{enumerate}
    \item If $\beta$ is an integer, $\frac{3}{2}<s$ and
    $\frac{2}{\beta-1}-\frac{3}{2}\leq \delta$, then
    \begin{equation}
      \label{eq:non:1}
      \|w^\beta\|_{H_{s-1,\delta+2}}\leq C_n 
\left(\|w\|_{H_{s,\delta}}\right)^\beta.
    \end{equation}
    \item If $\beta\not\in \setN$,
    $\frac{5}{2}<s<\beta-[\beta]+\frac{5}{2}$ and
    $\frac{2}{[\beta]-1}-\frac{3}{2}\leq \delta$, then
    \begin{equation}
      \label{eq:non:2}
      \|w^\beta\|_{H_{s-1,\delta+2}}\leq C_n 
\left(\|w\|_{H_{s,\delta}}\right)^{[\beta]}.
    \end{equation}
  \end{enumerate}
\end{prop}
\begin{rem}[Convention about constants]
  \label{rem:section-0005b:1}
  We have denoted the constant in this proposition explicitly by
  $C_n$, we will do the same for some other inequalities, because it
  comes in handy in the proof of the main theorem. 
  However in the rest of the paper we will denote constants by the
  generic letter $C$.
\end{rem}

In order that $\delta$ will belong to the range of isomorphism, we need that 
$(-\frac{3}{2},-\frac{1}{2})\cap [\frac{2}{[\beta]-1}-\frac{3}{2},
\infty)\neq\emptyset$. 
Taking into account that $\beta=\frac{2}{\gamma-1}$,  that gives
$1<\gamma < \frac{5}{3}$.
\begin{proof}
  If $\beta$ is an integer, then we apply Lemma \ref{lem:3} with
  $u_i=w$, $i=1,\ldots,\beta$. 
  That requires that $(\delta+2)\leq \beta\delta+(\beta-1)\frac{3}{2}$
  and $\frac{3}{2}<s$ and hence we get (\ref{eq:non:1}). 
  For the second part we set $\sigma=\beta-[\beta]+1$, then we apply
  Lemma \ref{lem:3} with $m=[\beta]$, $u_i=w$ for
  $i=1,\ldots,[\beta]-1$ and $u_m=w^\sigma$, and get that
  \begin{equation}
    \label{eq:non:3}
    \|w^\beta\|_{H_{s-1,\delta+2}}\leq 
    C\left(\|w\|_{H_{s,\delta}}\right)^{[\beta]-1}
    \|w^\sigma\|_{H_{s-1,\delta}},
  \end{equation}
  provided that $(\delta+2)\leq [\beta]\delta+([\beta]-1)\frac{3}{2}$. 
  Now by Kateb's estimate in the weighted spaces, Proposition
  \ref{Kateb}, we have that for $\frac{3}{2}<s-1<\sigma+\frac{1}{2}$,
  \begin{equation*}
    \|w^\sigma\|_{H_{s-1,\delta}}\leq C \|w\|_{H_{s-1,\delta}}\leq C 
    \|w\|_{H_{s,\delta}}.
  \end{equation*}
  Inserting it in (\ref{eq:non:3}) we get (\ref{eq:non:2}) with $C_n=C^2$.
  
\end{proof}

For the finite mass we need by Proposition \ref {prop: L-1} that
$\rho$, which is given by (\ref{eq:non:6}), belongs to
$H_{s',\delta'}$, for some $\delta'>\frac{3}{2}$ and 
$0\leq s'\leq s$.
\begin{prop}
  \label{prop:power:2}
  Suppose $w\in H_{s,\delta}$, $w\geq 0$, $2\leq \beta$, and
  $\frac{5}{2}<s$ if $\beta$ is an integer and
  $\frac{5}{2}<s<\frac{5}{2}+\beta-[\beta]$ otherwise. 
  If $\frac{3}{[\beta]}-\frac{3}{2}<\delta$, then
  \begin{equation}
  \label{eq:non:4}
    \|w^\beta\|_{H_{s-1,\delta'}}\leq C 
    \left(\|w\|_{H_{s,\delta}}\right)^{[\beta]}
  \end{equation}
  for some $\delta'>\frac{3}{2}$.
\end{prop}

\begin{proof}
 The proof is similar to the previous proposition. If $\beta$ is an integer, 
then we apply Lemma \ref{lem:3} with $u_i=w$, $i=1,\ldots,\beta$ and we get 
(\ref{eq:non:4}) under the condition $\frac{3}{2}<\delta'\leq 
\beta\delta+(\beta-1)\frac{3}{2}$. In case $\beta$ is not an integer, then we 
set $\sigma=\beta-[\beta]+1$ and apply Lemma \ref{lem:3} with $u_i=w$, 
$i=1,\ldots,[\beta]-1$, and $u_m=w^\sigma$, and  with the combination of 
Proposition  \ref{Kateb}, we obtain (\ref{eq:non:4}) under the condition 
$\frac{3}{2}<\delta'\leq [\beta]\delta+([\beta]-1)\frac{3}{2}$. So in both 
cases, we have the condition $\frac{3}{[\beta]}-\frac{3}{2}<\delta$.
 
\end{proof}

Note that in order that $\delta$ will be in the range of isomorphism, we
require that $\frac{3}{[\beta]}-\frac{3}{2}<-\frac{1}{2}$. 
This implies that $\frac{2}{\gamma-1}=\beta\geq[\beta]>3$, or 
$1<\gamma<\frac{5}{3}$.

\subsection{Difference estimates of powers}
\label{diff-estimates}
We encounter the following  difficulty concerning the  $L^2_{\delta}$
difference estimate. Namely, by inequality (\ref{eq:Poisson:3}), 
(\ref{eq:non:6}) and Proposition
\ref{equivalence}
 \begin{equation}
  \label{eq:difference:1}
  \begin{split}
    \|\nabla\phi_1-\nabla\phi_2\|_{L^2_\delta}\leq C
    \|\nabla\phi_1-\nabla\phi_2\|_{H_{1,\delta+1}}\\
    \leq  C C_e \|\rho_1-\rho_2\|_{H_{0,\delta+2}}= C C_e c_{K,\gamma}
    \|w_1^\beta-w_2^\beta\|_{L^2_{\delta+2}},
  \end{split}
\end{equation}
where $\beta=\frac{2}{\gamma-1}$. 
The problem  is that in (\ref{eq:difference:1}) we have
a difference in the $L^2_{\delta+2}$ norm, while in (\ref{eq:lower}) we
need the $L^2_{\delta}$ norm.
To overcome this problem we shall use a embedding property as given by
Proposition \ref{prop:emb-cont}.
\begin{prop}[Nonlinear estimate for the differences of two solutions]
  \label{prop:section-0005b:1}
  Under the condition of Proposition \ref{prop:power} the following
  estimate holds
  \begin{equation}
    \label{eq:section-0005b:1}
    \|w_1^\beta-w_2^\beta\|_{L^2_{\delta+2}}\leq C_d \|w_1-w_2\|_{L^2_{\delta}}
  \end{equation}
  where the constant $C_d^2\leq 
C\frac{\beta^2}{2}\left(\|w_1\|_{H_{s,\delta}}^{2(\beta-1)}+\|w_2\|_{H_{s,\delta
} }^{2(\beta-1)}\right)$.
\end{prop}
\begin{proof}
  We first write the difference in a integral form
  \begin{equation*}
    (w_1^\beta-w_2^\beta)=\int_0^1\beta\left(tw_1+(1-t)w_2\right)^{\beta-1}
    \left(w_1-w_2\right)dt
  \end{equation*}
  Note that $0\leq w_1,w_2 $ and $1\leq \beta-1$, so by using the convexity of 
the
  function $t^{\beta-1}$ we get that
  \begin{equation}
    \label{eq:non:7}
    \begin{split}
      |w_1^\beta-w_2^\beta| \leq &
      \int_0^1\beta\left(tw_1+(1-t)w_2\right)^{\beta-1} |w_1-w_2|dt\\
      \leq & \int_0^1\beta\left(tw_1^{\beta-1}+(1-t)w_2^{\beta-1}\right) 
|w_1-w_2|dt
      \\ \leq & \frac{\beta}{2}\left(w_1^{\beta-1}+w_2^{\beta-1}\right) 
|w_1-w_2|.
    \end{split}
  \end{equation}
  As before we start considering the case $\beta\in \setN$,  then by
(\ref{eq:non:7})
  \begin{equation}
  \label{eq:difference:2}
    \begin{split}
      & \|w_1^\beta-w_2^\beta\|_{L_{\delta+2}}^2= \int
      (1+|x|)^{2(\delta+2)}|w_1^\beta-w_2^\beta|^2 dx\\ \leq &
      \frac{\beta^2}{2}\int
      (1+|x|)^{2(\delta+2)}\left(w_1^{2(\beta-1)}+w_2^{2(\beta-1)}\right)
      |w_1-w_2|^2 dx \\ = & \frac{\beta^2}{2}\int      
\frac{(1+|x|)^{2(\delta+2)}}{(1+|x|)^4}\left(\left((1+|x|)^{\frac{2}{\beta-1}}
w_1
        \right)^{2(\beta-1)}+\left((1+|x|)^{\frac{2}{\beta-1}}w_2
        \right)^{2(\beta-1)}\right) |w_1-w_2|^2 dx \\ \leq &      
\frac{\beta^2}{2}\left(\left(\|w_1\|_{L_{\frac{2}{\beta-1}}^\infty}\right)^{
          2(\beta-1)}+\left(\|w_2\|_{L_{\frac{2}{\beta-1}}^\infty}\right)^{
          2(\beta-1)}\right)\int (1+|x|)^{2\delta}|w_1-w_2|^2dx \\
      \leq&      
\frac{\beta^2}{2}\left(\left(\|w_1\|_{L_{\frac{2}{\beta-1}}^\infty}\right)^{
          2(\beta-1)}+\left(\|w_2\|_{L_{\frac{2}{\beta-1}}^\infty}\right)^{
          2(\beta-1)}\right)\|w_1-w_2\|_{L^2_{\delta}}^2.
    \end{split}
  \end{equation}
  Since $\frac{2}{\beta-1}\leq \frac{3}{2}+\delta$, we get by
  Proposition \ref{prop:emb-cont} (i) that
  \begin{equation*}
    \|w_i\|_{L_{\frac{2}{\beta-1}}^\infty}\leq C \|w_i\|_{H_{s,\delta}}, \qquad 
i=1,2.
  \end{equation*}
  In the case that $\beta\not\in\setN$, we replace
  $(1+|x|)^{\frac{2}{\beta-1}}$ by $(1+|x|)^{\frac{2}{[\beta]-1}}$ in inequality
  (\ref{eq:difference:2}). 
  Since $1\leq\frac{\beta-1}{[\beta]-1}$,
  $(1+|x|)^{\frac{4(\beta-1)}{[\beta]-1}}\leq (1+|x|)^4$, and hence we can 
proceed as in the
  case that $\beta$ is an integer.
\end{proof}

\section{Proof of the main results}
\label{sec:makin-fixp-argum}
We have explained  the main idea of the proof in section
\ref{sec:struct-proof-organ}.
We start with the construction of the map $\Phi$ which we will use for the
fixed--point theorem.
\subsection{Construction of the map $\Phi$}
\label{sec:construction-phi-1}
For a given $w(x,t)$, let
\begin{equation*}
  \widehat w = \Phi(w)
\end{equation*}
where $\Phi$ is constructed as follows.
\begin{enumerate}[A.)]
  \item \label{item:proto-euler-poisson-banach:1} The inversion of the
  Makino variable: The Makino variable
  $w=\frac{2\sqrt{K \gamma}}{\gamma-1}\rho^{\frac{\gamma-1}{2}}$ is a nonlinear
  function of the density $\rho$. 
  We need to consider $\rho= c_{K\gamma}w^{\frac{2}{\gamma-1}}$ as a function of
  $w$, where $c_{K\gamma}$ is given in (\ref{eq:non:6}). 
  However, due to the fact that
  $\Delta: H_{s,\delta}\to H_{s-2,\delta+2}$ is an isomorphism (Theorem
  \ref{thr:proto-euler-poisson-banach:1}), we need to assure that the
  source term $\rho$ belongs to $H_{s-2,\delta+2}$. 
  We show this by using the nonlinear estimates (\ref{eq:non:1}) and
  (\ref{eq:non:2}).
  \item The elliptic step:
  \label{item:iteration-scheme:2} With the $\rho$ from the last step, we
  construct, $\phi$ (resp.\ $\nabla \phi$) as a solution of the Poisson equation
  (\ref{eq:iteration-scheme:3}).
  (See section \ref{sec:elliptic-estimate}).
  \vskip 2mm
  \item The hyperbolic step: We now construct $\widehat w$. 
  We chose initial data in accordance with the assumptions made in
  Theorem \ref{thm:main:1}. 
  Then we cast the Euler equations into symmetric hyperbolic form
  (\ref{eq:eulerpoisson:1})
  and solve it with these initial data and the external source term
  $\nabla\phi$ which we constructed in the last step. 
  This results in $\widehat w$.
\end{enumerate}
For convince we write the map $\Phi$ as
\begin{equation}
  \label{eq:fixpointscheme:62}
  \Phi = \Phi_1 \circ \Phi_2 \circ \Phi_3 : w
  \underset{\mbox{inverse}}{\mapsto} \rho  \underset{\mbox{ellp}}{\mapsto}
  \nabla\phi 
  \underset{\mbox{hyp}}{\mapsto} (\widehat w, \widehat v^a).
\end{equation}

For that map we have to show that it maps a closed bounded set into
itself and that it is a contraction.
We start with the construction of appropriate sets of functions in
which our map $\Phi$ will act upon.
Denote $\beta=\frac{2}{\gamma-1}$ and let $s,\delta,\gamma$ satisfy the conditions of
Theorem \ref{thm:main:1}. 
Let us for the moment assume that $\beta\in\setN$, the case
$\beta\not \in \setN$ is very similar but we leave it out for the convenience of
the reader.
We chose $M_0$ such that the initial data satisfy:
\begin{equation}
  \label{eq:section-0006:6}
  \|(w_0,v^a_0)\|_{H_{s,\delta}}\leq M_{0}.
\end{equation}
Let  $B_h$  be a closed bounded set  given by
\begin{equation*}
  B_h=\{w \in C\left(\lbrack0,T\rbrack ; {H_{s,\delta}}\right) : 0\leq w, \ 
  w(0,x)=w_0(x), \sup_{0\leq t\leq
    T} \|w(t,\cdot)\|_{H_{s,\delta}} \leq 2M_0 \}.
\end{equation*}
 
\subsection{The map $\Phi$ is a self--map}
\label{sec:boundness-phi}

Now we are in a position to show that $\Phi$ maps $B_h$ to $B_h$.
We prove it as described in (\ref{eq:fixpointscheme:62}) step
by step.
So take $w\in B_{h}$:
\begin{enumerate}[A.)]
  \item \textbf{The inversion of the Makino variable.}
 Let $\rho=\Phi_{1}(w)$, define by 
  $\rho(t,x)=c_{K,\gamma}w^{\beta}(t,x)$. 
  Then by the power estimates (\ref{eq:non:1}) of Proposition
  \ref{prop:power} in section \ref{sec:nonlinear-estimate} we obtain
  an estimate of the form
  \begin{equation}
    \label{eq:fixpoint-scheme:47}
    \|\rho(t,\cdot)\|_{H_{s-1,\delta+2}}\leq C_n
    \left(c_{K,\gamma}\|w(t,\cdot)\|_{H_{s,\delta}}\right)^\beta.
  \end{equation}
\vskip 2mm
  \item \textbf{The Elliptic step:}
  \label{item:proto-euler-poisson-banach:3}
   Let $\Phi_2(\rho)=\nabla\Delta^{-1}\rho$.
  Since $\rho\in H_{s-1,\delta+2}$ from the previous step, $\phi$
  (resp.\ $\nabla \phi$) is constructed via the Poisson equation,
  (\ref{eq:iteration-scheme:3}), by applying Theorem
  \ref{thr:proto-euler-poisson-banach:1} that provides the solution to
  it, and by inequality (\ref{eq:Poisson:3}) we obtain
  \begin{equation*}
    \|\nabla\phi(t,\cdot)\|_{H_{s,\delta+1}}\leq C_{e} 
\|\rho(t,\cdot)\|_{H_{s-1,\delta+2}}.
  \end{equation*}
 Combining it with the previous step we obtain
  \begin{equation}
    \label{eq:fixpoint-scheme:16}
    \begin{split}
    \left\| \nabla\phi(t,\cdot) \right\|_{H_{s,\delta+1}}^{2}
        & \leq        C_{e}^2      \left\| 
\rho(t,\cdot)\right\|_{H_{s-1,\delta +2}}^{2} 
    \leq  C_{e}^2 C_{n}^2 
c_{K,\gamma}^{2\beta}\left\|w(t,\cdot)\right\|_{H_{s,\delta}}^{2\beta}
  \\ &    \leq C_{e}^2C_{n}^2c_{K,\gamma}^{2\beta}(2M_0)^{2\beta}.
      \end{split}
  \end{equation}
\vskip 2mm
  \item \textbf{The Hyperbolic step:}
 Let $\Phi_3:\nabla\phi \mapsto (\widehat w,\widehat v^a)$, where
  $(\widehat w, \widehat v^a)$ denote the solution of the following
  system
  \begin{equation}
  \label{eq:eulerpoisson:1}
    \begin{pmatrix}
      1 & 0
      \\
      0 & \delta_{ab}
    \end{pmatrix}
    \partial_t
    \begin{pmatrix}
      \widehat w
      \\
      \widehat v^b
    \end{pmatrix}
    +
    \begin{pmatrix}
      \widehat v^c & \frac{\gamma -1}{2}\delta^c_b
      \\
      \frac{\gamma -1}{2}\delta^c_a & \delta_{ab}\widehat v^c
    \end{pmatrix}
    \partial_c
    \begin{pmatrix}
      \widehat w
      \\
      \widehat v^b
    \end{pmatrix}
    =
    \begin{pmatrix}
      0\\
      -\partial_a \phi\\
    \end{pmatrix},
  \end{equation}
  with the given initial data $(w_0,v^a_0)$ which satisfy
  (\ref{eq:section-0006:6}).
  These initial data and the source term $\nabla \phi$ of the last step
  satisfy the conditions of Theorem \ref{thm:existence-quasi}. 
  Hence we obtain a solution
  $U=(\widehat w, \widehat v^a)\in C([0,T],H_{s,\delta})\cap
  C^1([0,T],H_{s-1,\delta+1})$.
  Now using Corollary \ref{cor:1} and estimate \ref{eq:energy:1} we
  obtain
  \begin{equation*}
    \|U(t)\|_{H_{s,\delta}}^2\leq e^{C_1 t}\left(M_0^2+ \int\limits_0^{t}
      \|F(\tau,\cdot)\|_{H_{s,\delta}}^2 d\tau\right),
  \end{equation*}
  where $F(t,x)=(0,\nabla \phi(t,x))$.
  Now using the fact that
  $\| \nabla \phi \|_{H_{s,\delta}}\leq \| \nabla \phi
  \|_{H_{s,\delta+1}}$,
  inequalities (\ref{eq:fixpoint-scheme:16}) and
  (\ref{eq:fixpoint-scheme:47}) we obtain
  \begin{equation*}
    \displaystyle\sup_{0\leq t\leq    T}\|\Phi(w(t))\|_{H_{s,\delta}}^2
    \leq    \sup_{0\leq t\leq T}    \|U(t)\|_{H_{s,\delta}}^2
    \leq e^{C_1 T} \left[M_0^2 + 
C_{e}^2C_{n}^2c_{K,\gamma}^{2\beta}(2M_0)^{2\beta}T\right]. 
  \end{equation*}
  So choosing $T$ sufficiently small that we obtain the following
  inequality 
  \begin{equation*}
    \displaystyle\sup_{0\leq t\leq    T}\|U(t)\|_{H_{s,\delta}}^2 \leq 4M_0^{2}.
  \end{equation*}
  From which follows that $\widehat{w}=\Phi(w)\in B_h$ and that $\Phi$ maps 
$B_h$ into $B_h$.
  During the course of this proof, we have to use the fact that
  $0\leq \widehat w$, given that $0\leq w_0$. 
  That this is, in fact, true, can be seen easily by integrating the
  continuation equation along their characteristics.
  For details, we refer to Makino
  \cite[p.~467]{makino86:_local_exist_theor_evolut_equat_gaseous_stars}.
  \hfill $\Box$
\end{enumerate}

\subsection{The map $\Phi$ is a contraction in $L^2_{\delta}$}
\label{sec:map-phi-contraction}
The proof of the contraction combines the energy estimates in the
$L^2_\delta$ spaces with the nonlinear estimate of the difference which
we obtained in subsection \ref{diff-estimates} and the inequalities of
the previous steps.   
 
Let $w_1,\ w_2\in B_h$, then
\begin{align}
  &\|\Phi(w_1(t))- \Phi(w_2(t))\|_{L_{\delta}^2}^{2}
    =
    \|\widehat w_1(t)-\widehat  w_2(t)\|_{L_{\delta}^2}^{2}\nonumber \\
  &\leq
    e^{C_2t}  \int\limits_0^t 
    \|\nabla \phi_1(\tau) -  \nabla \phi_2(\tau)\|_{L_{2}^2}^{2} d\tau
    \quad\mbox{by eq.
    (\ref{eq:lower}) of  Corollary \ref{cor:2} }\nonumber\\
  &\leq
    e^{C_2t}  \int\limits_0^t 
    \|\nabla \phi_1(\tau) -    \nabla \phi_2(\tau)\|_{H_{1,\delta+1}}^{2}d\tau
    \quad\mbox{by  (\ref{eq:norm:3}) and Proposition \ref 
    {embedding1}}\nonumber\\
  &\leq
    e^{C_2t} C_{e}^2  \int\limits_0^t 
    \|\rho_1(\tau)-\rho_2(\tau)\|_{H_{0,\delta+2}}^{2}d\tau
    \quad\mbox{by eq.
    (\ref{eq:Poisson:3})}\nonumber\\
  &\leq
    e^{C_2t} C_{e}^2c_{K,\gamma}^2  \int\limits_0^t 
    \|w_1^\beta(\tau)- w^\beta_2(\tau)\|_{L^2_{\delta+2}}^2d\tau
    \quad\mbox{by  (\ref{eq:norm:3}) and Proposition \ref 
    {embedding1}}\nonumber\\
  &\leq
    \label{eq:section-0006:5}
    e^{C_2t} C_{e}^2c^2_{K,\gamma}C_d^2 \int\limits_0^t 
    \|w_1(\tau)-
    w_2(\tau)\|_{L_{2,\delta}}^{2} d\tau
    \quad\mbox{by eq.
    (\ref{eq:section-0005b:1})}.
\end{align}
Now taking the $\sup$--norm of (\ref{eq:section-0006:5}), we obtain
\begin{equation*}
  \displaystyle\sup_{0\leq t\leq T}
  \|\Phi(w_1(t))- \Phi(w_2(t))\|_{L_{2,\delta}}^{2}
  \leq 
  e^{C_2T} T \left( C_{e}c_{K,\gamma}C_d\right)^2
  \displaystyle\sup_{0\leq t\leq T}  \|w_1(t)- w_2(t)\|_{L_{2,\delta}}^{2}.
\end{equation*}
Now taking $T$ sufficiently small so that we have
$e^{C_2T} T \left( C_{e}c_{K,\gamma}C_d\right)^2<1$, then $\Phi$ is
indeed a contracting map.

\begin{proof}[Proof of Theorem  \ref{thm:main:1}]
We have shown that $\Phi$ maps $B_h \subset H_{s,\delta}$ to $B_h$
and that it is a contraction with respect to the $L_{\delta}^2$--norm. 
By Theorem \ref{thr:fixpoint-banach-norms:1} the map $\Phi$ has a unique
fixed--point $w^{\star}$ in $H_{s,\delta}$. 
However, in order not to have a clumsy notation we drop the ${}^{*}$
and $\ {}\widehat {}\ $ in the following.
The vector valued function $U=(w, v^a)$ is the solution to
the Euler--Poisson-Makino system
(\ref{eq:iteration-scheme:1b})-(\ref{eq:iteration-scheme:3b}) and it
belongs to $H_{s,\delta}$. 
Since $U$ solves the symmetric hyperbolic
system (\ref{eq:neu-existence:21}), we conclude by Theorem
\ref{thm:existence-quasi} that
\begin{equation*}
  U=(w,  v^a)\in C([0,T],H_{s,\delta})\cap 
  C^1([0,T],H_{s-1,\delta+1}).
\end{equation*}
At the beginning of the proof we set $\beta\in\setN$, now in the case
of $\beta\notin \setN$ we would have used the estimate
(\ref{eq:non:2}) instead (\ref{eq:non:1}) of the same proposition. 
However the rest of the proof would not have been altered.
This completes the proof of Theorem \ref{thm:main:1}.
\end{proof}

We turn now to the proof of Corollary \ref{cor:section-0003:2}.

\begin{proof}[Proof of Corollary \ref{cor:section-0003:2}]
  Let $(w,v^a)$ be the solution to the Euler--Poisson--Makino system
  (\ref{eq:iteration-scheme:1b})--(\ref{eq:iteration-scheme:3b}), then
  $\rho= c_{K,\gamma}w^{\frac{2}{\gamma-1}}$ is the density. 
  By Proposition \ref {prop:power:2} $\rho \in H_{s-1,\delta'}$ for some
  $\delta'>\frac{3}{2}$. 
  Hence by Propositions \ref{monoton} and \ref{prop: L-1},
  \begin{equation*}
    \|\rho\|_{L^1}\leq C\|\rho\|_{L^2_{\delta'}}\leq 
    C\|\rho\|_{H_{s-1,\delta'}}.
  \end{equation*}
  We turn now to the energy functional (\ref{eq:energy-functional}). 
  Note that $(w,v^a)\in L^\infty$ by the Sobolev embedding in the
  weighted spaces, Proposition \ref{prop:emb-cont}, and that
  $\rho^\gamma=c_{K,\gamma}^{\gamma+1} w^2\rho$. 
  Hence, the first two terms of (\ref{eq:energy-functional}) are
  finite since $\rho\in L^1$. 
  Set
  \begin{equation*}
    V(t,x)=\int \frac{\rho(t,y)}{|x-y|}dy=\int_{\{|y-x|\leq 1\}} 
    \frac{\rho(t,y)}{|x-y|}dy+\int_{\{|y-x|> 1\}} 
    \frac{\rho(t,y)}{|x-y|}dy.
  \end{equation*}
  Then for $t\in [0,T]$,
  \begin{equation*}
    |V(t,x)|\leq 2\pi 
    \|\rho(t,\cdot)\|_{L^\infty}+\|\rho(t,\cdot)\|_{L^1}.
  \end{equation*}
  Thus $V(t,\cdot)\in L^\infty$, which implies that
  \begin{equation*}
    \iint\frac{\rho(t,x)\rho(t,y)}{|x-y|}dx dy\leq \int 
    V(t,x)\rho(t,x)dx\leq 
    \|V(t,\cdot)\|_{L^\infty} \|\rho(t,\cdot)\|_{L^1}.
  \end{equation*}
 
\end{proof}

\begin{appendices}

\section{The modified Banach fixed--point theorem}
\label{sec:proof-modif-banach}

\begin{thm}
  \label{thr:fixpoint-banach-norms:1}
  Let $X$ and $Y$ be two Hilbert spaces such that $ X\subset Y$,
  $\left\| \, \cdot \, \right\|_X$ and $\left\|\,\cdot\, \right\|_Y$ denote their
  norms, $\mathit{K}\subset X$ be a closed bounded set and let
  $\Phi:X\to X$ be a map such that
  \begin{enumerate}
    \item $\Phi$ maps $\mathit{K}$ into $\mathit{K}$, that is,
    $\Phi(x)\in \mathit{K}$ for all $x\in \mathit{K}$;
    \item $\Phi$ is a contraction map in $Y$, that is, there exists a
    constant $0<\Lambda<1$ such that
    \begin{equation*}
      \|\Phi(x_1)-\Phi(x_2)\|_{Y} \leq \Lambda \|x_1-x_2\|_{Y} \qquad \text{for all  }\       x_1,x_2\in \mathit{K}.
    \end{equation*}
  \end{enumerate}
  Then $\Phi$ admits a unique fixed--point $x^{\star}\in \mathit{K}\cap X$ such 
that
  \begin{math}
    \Phi(x^{\star}) = x^{\star}.
  \end{math}
\end{thm}
Although this theorem seems to be part of the mathematical folklore,
we failed to find a proof of it and that is why, and for the
convenience of the reader, we present the proof in the following.
\begin{proof}
  Let $x_0 \in \mathit{K}$ and define a sequence $\{x_n\}$ by
  \begin{math}
    x_n = \Phi(x_{n-1}). 
  \end{math}
  It is straightforward to show
  \begin{math}
    \|x_{n+1}- x_n\|_Y \leq \Lambda^n \|x_1- x_0\|_{Y}.
  \end{math}
  Hence $\{x_n\}$ is a Cauchy sequence in $Y$, and therefore it
  converges strongly to a limit $x^{\star}$ in $Y$.
  Moreover, $x^{\star}$ is the only fixed--point.
  
  Since $\{x_n\}\subset \mathit{K}$, it is a bounded sequence in $X$, and
  since $X$ is a Hilbert space, the Banach--Alaoglu theorem implies
  that there is a subsequence $\{x_{n_k}\}$, which converges weakly to
  $\hat{x}\in \mathit{K}\cap X$. 
  It remains to show that $\{x_{n_{k}}\}$ converges weakly in $Y$. 
  This implies that $\hat{x}=x^{\star}$ and hence $x^{\star}$ belongs to
  $ X$.
  So let $X^{\prime}$ and $Y^{\prime}$ denote the dual spaces. 
  Weak convergence means that
  \begin{equation*}
    f\left( x_{n_k}\right) \to f\left(\hat{x}\right) \qquad \text{for all } \ 
    f\in 
    X^{\prime}.
  \end{equation*}
  Since $X\subset Y$, $Y^{\prime}\subset X^{\prime}$, and hence
  \begin{equation*}
    f\left( x_{n_k}\right) \to f\left(\hat{x}\right) \qquad \text{for all } \ 
    f\in 
    Y^{\prime}.
  \end{equation*}
  Thus $\{x_{n_k}\}$ converges weakly in $Y$ and that completes the
  proof.
\end{proof}
\end{appendices}

\bibliographystyle{amsalpha-url}

\bibliography{jde}        

\end{document}